\documentclass[11pt]{amsart}
\usepackage{amsmath,amsthm, amscd, amssymb, amsfonts}

\usepackage[]{color} 

\usepackage{graphicx} 

\usepackage{mathrsfs} 
\usepackage{figsize}
\usepackage{algorithm}
\usepackage{algpseudocode}
\usepackage{enumerate} 

\usepackage{float} 

\usepackage{tikz} 

\usepackage{tkz-berge} 
\usepackage{tkz-graph}
\usetikzlibrary{arrows}
\usetikzlibrary{calc}
\usetikzlibrary{matrix}
\usetikzlibrary{decorations.pathreplacing}  

\usepgflibrary{shapes.geometric} 
\usetikzlibrary[topaths]
\usetikzlibrary{positioning}
\usetikzlibrary{arrows}  
\usetikzlibrary{graphs,graphs.standard}

\usetikzlibrary{arrows.meta,calc}

\usepackage{tabularx}

\theoremstyle{plain}
\newtheorem{theorem}{Theorem}[section]
\newtheorem{corollary}[theorem]{Corollary}
\newtheorem{proposition}[theorem]{Proposition}
\newtheorem{lemma}[theorem]{Lemma}
\newtheorem{definition}{Definition}[section]  
 
\newtheorem{remark}[theorem]{Remark}
\newtheorem{example}[theorem]{Example}

\usepackage[colorlinks]{hyperref}

\def\F{\mathbb{F}}

 \def\ad{\mathrm{ad}}

\title[The comaximal Graph]{The comaximal Graph of a finite-dimensional Lie Algebra}

\thanks{...}

\author{David A. Towers}
\address{Lancaster University, School of Mathematical Sciences, Lancaster, LA1 4YF, UK} 
  \email{d.towers@lancaster.ac.uk}

\author{Yesneri Zuleta}
 \address{Universidad del Norte,  Departamento de Matemáticas y Estadística, Km 5 via a Puerto Colombia, Barranquilla, Colombia.}
  \email{zuletay@uninorte.edu.co}

\author{Ismael Gutierrez}
 \address{Universidad del Norte,  Departamento de Matemáticas y Estadística, Km 5 via a Puerto Colombia, Barranquilla, Colombia.}
  \email{isgutier@uninorte.edu.co }

\subjclass[2010]{17B30, 17B45, 05C40, 05C25}
\keywords{Graphs associated with Lie Algebras; nilpotent Lie Algebra; hypercenter of a Lie Algebra; nilpotentizer; strongly self-centralizing Lie subalgebras}

\begin{document}

\begin{abstract}
In this paper, we introduce the comaximal graph $\Gamma(L)$ of a finite-dimensional Lie algebra $L$, whose vertices are the nontrivial proper Lie subalgebras of $L$ over a field $\mathbb{F}$, and two vertices $A$ and $B$ are adjacent if and only if $\langle A, B\rangle =L$. We establish general structural properties, including a characterization of isolated vertices via the Frattini subalgebra and a criterion for completeness in terms of $\mu$-algebras. We classify $\Gamma(L)$ for all Lie algebras of dimension at most three over a finite field $\mathbb{F}_q$, providing an explicit description in each case. The resulting graphs exhibit a rich range of behaviors, depending on the structure of the derived algebra and the action of $\operatorname{ad}x$. For $L\cong \mathfrak{sl}_2(\mathbb{F}_q)$, we determine several graph invariants, including the degree sequence, clique number, chromatic number, domination number, diameter, and radius, and show that $\Gamma(L)$ is connected and non-planar. The graph contains a large clique formed by the nonsplit semisimple lines together with the Borel subalgebras, while the nilpotent and split semisimple lines have a more restricted adjacency structure governed by their containment in Borel subalgebras. 
\end{abstract}

\maketitle


\section{Introduction}

Associating graphs with algebraic structures has proven to be a fruitful approach for encoding algebraic information into combinatorial form. Well-known examples include commuting graphs, non-commuting graphs, intersection graphs, and power graphs of groups and rings. These constructions often reveal subtle structural features that are not immediately visible from the algebraic definition alone.

In the context of vector spaces, A. Das \cite{Das1} introduced the \emph{subspace sum graph}, whose vertices are the nonzero proper subspaces of a finite-dimensional vector space $V$, and where two subspaces are adjacent if and only if their sum equals $V$. This construction was later studied and extended by several authors \cite{Chelliah, Bilal}.
The adjacency condition reflects a spanning property rather than inclusion, leading to a highly symmetric and well-understood graph.

More recently, the systematic study of graphs associated with finite-dimensional Lie algebras has been initiated in \cite{TGF} and \cite{TGF2}. In \cite{TGF}, the nilpotent graph of a Lie algebra $L$ is introduced, whose vertices are the non-nilpotent elements of $L$, with two vertices adjacent if and only if they generate a nilpotent subalgebra. In \cite{TGF2}, an analogous construction is carried out for solvability, defining the solvable graph of L. Both works establish structural properties of these graphs, including connectedness, diameter, and clique number, and provide classifications over finite fields. The present paper continues this line of investigation by introducing a graph that reflects the subalgebra structure of $L$ rather than the properties of individual elements.

Let $L$ be a finite-dimensional Lie algebra over a field $\mathbb{F}$. The \emph{comaximal graph} of $L$, denoted by $\Gamma(L)$, is the simple graph whose vertex set consists of all nontrivial proper Lie subalgebras of $L$, and where two distinct vertices $A$ and $B$ are adjacent if and only if $\langle A, B\rangle = L$.

Unlike the subspace sum graph, the vertex set of $\Gamma(L)$ is restricted to Lie subalgebras rather than arbitrary subspaces. Consequently, the structure of $\Gamma(L)$ depends not only on the dimension of $L$ but also on its internal Lie structure. Even in low dimensions, the behavior of $\Gamma(L)$ varies dramatically among abelian, nilpotent, solvable, and simple Lie algebras.

Throughout this paper, we focus primarily on Lie algebras over finite fields $\mathbb{F}_q$. This setting is natural for a combinatorial study: over finite fields, the number of subalgebras is finite and explicitly computable, the subalgebra lattice has a concrete projective-geometric interpretation, and the resulting graphs can be described with precise cardinalities and invariants.

Our aim is to initiate a systematic study of the comaximal graph of finite-dimensional Lie algebras, with a primary focus on Lie algebras over finite fields. We begin by establishing general structural properties of $\Gamma(L)$,
including a characterization of isolated vertices in terms of the Frattini subalgebra and a criterion for completeness.

The main part of the paper is devoted to a complete classification of $\Gamma(L)$ for Lie algebras of dimension at most three over a finite field $\mathbb{F}_q$. When $\dim L \le 2$, the graph is easily determined. The three-dimensional case is substantially richer and splits naturally according to the dimension of the derived algebra $L'$. 

In the simple case $L=\mathfrak{sl}_2(\mathbb{F}_q)$, we compute the degree sequence of $\Gamma(L)$ explicitly and determine several graph invariants, including the clique number, chromatic number, domination number, diameter, and radius. We prove that $\Gamma(L)$ is connected and non-planar. The graph contains a large clique formed by the nonsplit semisimple lines together with the Borel subalgebras; the Borel subalgebras alone form a complete subgraph $K_{q+1}$, while the nilpotent and split semisimple lines have a more restricted adjacency structure governed by their containment in Borel subalgebras.

\section{Preliminaries}
\subsection*{Basics on Graphs.}
A graph $\Gamma = (V, E)$ consists of a set of vertices $V$ and a set of edges $E\subseteq  \{\{u, v\}\mid u,v \in V, u\neq v\}$. All graphs in this paper are simple and undirected. We recall the following standard terminology; for a comprehensive reference, see \cite{Diestel, West, Bondy}. The \emph{order} and \emph{size} of $\Gamma$ are $|V|$ and $|E|$, respectively. The degree $\deg(v)$ of a vertex $v$ is the number of edges incident to $v$. A graph is regular if all vertices have the same degree. The degree sequence of $\Gamma$ is the list of degrees of all vertices in non-increasing order. A graph is \emph{complete} if every pair of distinct vertices is adjacent; the complete graph on $n$ vertices is denoted by $K_n$.

A path of length $k$ in $\Gamma$ is a sequence of distinct vertices $v_0, v_1,\ldots, v_k$ such that $\{v_i, v_{i+1}\} \in E$ for all $i$. The distance $d(u, v)$ between two vertices is the length of a shortest path connecting them, or $\infty$ if no such path exists. The eccentricity $\operatorname{ecc}(v)$ of a vertex $v$ is $\max\{d(v, u)\mid u\in V\}$. The diameter $\operatorname{diam}(\Gamma)$ is the maximum eccentricity, the radius $r(\Gamma)$ is the minimum eccentricity, and the center of $\Gamma$ is the set of vertices achieving the minimum eccentricity.

A graph $\Gamma$ is connected if there is a path between every pair of vertices. The girth of $\Gamma$ is the length of a shortest cycle, or $\infty$ if $\Gamma$ has no cycles. A clique is a set of pairwise adjacent vertices. The clique number $\omega(\Gamma)$ is the size of a largest clique. An independent set is a set of pairwise non-adjacent vertices. A proper coloring of $\Gamma$ is an assignment of colors to vertices such that no two adjacent vertices share a color. The chromatic number $\chi(\Gamma)$ is the minimum number of colors needed. A graph is called perfect if
$\chi(H) = \omega(H)$ for every induced subgraph $H$ of $\Gamma$.

A dominating set is a set $S\subseteq V$ such that every vertex not in $S$ is adjacent to some vertex in $S$. The domination number $\gamma(\Gamma)$ is the size of a smallest dominating set. A graph is planar if it can be drawn in the plane without edge crossings. By Kuratowski's theorem, $\Gamma$ is planar if and only if it contains no subdivision of $K_5$ or $K_{3,3}$ as a subgraph. A claw is the complete bipartite graph $K_{1,3}$. A graph is claw-free if it contains no induced subgraph isomorphic to a claw.

Throughout the paper, we use $\Gamma(L)^*$ to denote the subgraph obtained from $\Gamma(L)$ by removing the isolated vertices.

\subsection*{Basics on Lie Algebras.}
Let $L$ be a finite-dimensional Lie algebra over a field $\mathbb{F}$. A \emph{Lie algebra} is a vector space $L$ over $\mathbb{F}$, equipped with a bilinear operation $[\cdot,\cdot]: L \times L \to L$, called the \emph{Lie bracket}, satisfying the following axioms:
\begin{enumerate}
    \item \textbf{Alternating property:} $[x,x] = 0$ \ $\forall  x \in L$;
    \item \textbf{Jacobi identity:} $[x,[y,z]] + [y,[z,x]] + [z,[x,y]] = 0$ $\forall x, y, z \in L$.
\end{enumerate}
These imply \textbf{anticommutativity}: $[x,y] = -[y,x]$ for all $x,y \in L$.

A Lie algebra is said to be \emph{abelian} if $[x,y] = 0$ for all $x,y \in L$. Every one-dimensional Lie algebra is abelian. The \emph{dimension} of a Lie algebra is its dimension as a vector space over $\mathbb{F}$. A \emph{subalgebra} of $L$ is a subspace closed under the Lie bracket. If $A, B$ are subspaces of $L$, then $[A,B]$ denotes the subspace spanned by all $[a,b]$ with $a \in A$, $b \in B$. The \emph{centraliser} of an element $x\in L$ is $C_L(x)=\{y\in L \mid [x,y]=0\}$.

The Frattini subalgebra $F(L)$ of a Lie algebra $L$ is the intersection of all maximal subalgebras of $L$. It is the largest subalgebra of $L$ with the property that every subalgebra generated by $F(L)$ together with any other proper subalgebra is still proper. If $L$ has no maximal subalgebras, we set $F(L) = L$. 

A subalgebra $B$ of $L$ is called a Borel subalgebra if it is a maximal solvable subalgebra of $L$. In the case $\mathfrak{sl}_2(\mathbb{F}_q)$, all Borel subalgebras are conjugate and two-dimensional. An element $x\in L$ is called nilpotent if $\ad\, x$ is a nilpotent linear map, and semisimple if $\ad\, x$ is a semisimple linear map. In $\mathfrak{sl}_2(\mathbb{F}_q)$, a semisimple element is called split if its characteristic polynomial splits over $\F_q$, and nonsplit otherwise.

 Let $\mathbb{F}$ be a field. Standard examples include: 
\begin{itemize}
    \item $\mathfrak{gl}_n(\mathbb{F})$, the Lie algebra of all $n \times n$ matrices over $\mathbb{F}$ with bracket $[x,y] = xy - yx$;
   \item $\mathfrak{sl}_n(\mathbb{F})\subseteq \mathfrak{gl}_n(\mathbb{F})$, the subalgebra of traceless $n \times n$ matrices over $\mathbb{F}$, which is simple for $n\geq 2$. In particular, $\mathfrak{sl}_2(\mathbb{F}_q)$, the three-dimensional simple Lie algebra over a finite field, which serves as the central example of this paper, with standard basis $\{x, y, h\}$, where
   \[x =   \begin{pmatrix} 
0 & 1 \\ 0 & 0 
\end{pmatrix},
  \quad
  y =   \begin{pmatrix} 
  0 & 0 \\ 1 & 0 
  \end{pmatrix},
  \quad
  h =  \begin{pmatrix} 
  1 & 0 \\ 0 & -1 
  \end{pmatrix},\]
and Lie brackets $[h,x]=2x,\qquad [h,y]=-2y,\qquad [x,y]=h$. 
\end{itemize}

The \emph{derived series} of $L$ is defined recursively by $L^{(0)} = L$, and $L^{(k+1)} = [L^{(k)}, L^{(k)}]$.
Further, the \emph{lower central series} is given by $L^{1}=L$, and $L^{i+1}=[L,L^{i}]$. The derived algebra of $L$, usually denoted by $L'$ is $L^2 = [L, L]$. It is spanned by all products of the form $[x, y]$ with $x, y\in L$; $L$ is called perfect if $L^2 = L$. The Lie algebra $L$ is called \emph{solvable} if $L^{(r)} = 0$ for some $r$, and \emph{nilpotent} if $L^k = 0$ for some $k$. Every nilpotent Lie algebra is solvable, but the converse is not true.

The smallest nonabelian nilpotent Lie algebra is the Heisenberg algebra $H_{3}(\F_q)$, appearing in Section 4.3.  In that case, $L'$ is one-dimensional and central, and every $\ad\, x$ is nilpotent of index at most $2$.

Nilpotent Lie algebras have many nilpotent elements, but not conversely: simple Lie algebras such as $\mathfrak{sl}_2(\F_q)$
contain nilpotent elements (e.g.\ $x$ and $y$) yet are far from nilpotent
themselves.  The comaximal graph $\Gamma(L)$ reflects this distinction:
nilpotent Lie algebras tend to have larger families of subalgebras and
simpler adjacency patterns, while semisimple Lie algebras exhibit a more
rigid and intricate subalgebra geometry.

Following V. R. Varea \cite{varea}, a Lie algebra $L$ is called a $\mu$-algebra if every proper subalgebra of $L$ is one-dimensional. This notion plays a key role in characterizing when the comaximal graph $\Gamma(L)$ is complete; see Theorem \ref{p:gen} below.

\section{The comaximal graph of a Lie algebra}

\begin{definition}
Let $L$ be a finite-dimensional Lie algebra over a field $\mathbb{F}$. We define the \emph{comaximal graph} of $L$ as the undirected, simple graph $\Gamma(L) = (\mathcal{V, E})$ as follows: $\mathcal{V}$ is the set of all non-trivial proper subalgebras of $L$, and for $A, B\in \mathcal{V}$, $\{A, B\}\in \mathcal{E}$ if and only if $\langle A, B\rangle=L$. 
\end{definition}

\begin{remark}
The definition is inspired by the subspace-sum graph introduced by A. Das \cite{Das1} and subsequently extended by R. Venkatasalam and S. Chelliah \cite{Chelliah}, where vertices are all nonzero proper subspaces of a vector space $V$, and adjacency is defined by spanning the whole space. In our setting, the vertices are Lie subalgebras rather than arbitrary subspaces, and the resulting graph reflects the internal algebraic structure of $L$ much more
subtly.
\end{remark}

\begin{lemma}\label{isolated points}
Let $L$ be a finite-dimensional Lie algebra over a field $\mathbb{F}$. Then a vertex of $\Gamma(L)$ is isolated if and only if it is contained in the Frattini subalgebra of $L$. In particular, if $F(L) = 0$, then the graph $\Gamma(L)$ has no isolated vertices.
\end{lemma}

\begin{proof}
Suppose that $A$ is an isolated vertex of $\Gamma(L)$, and let $M$ be any maximal subalgebra of $L$. If $A \not\subseteq M$, then maximality of $M$ implies $\langle A, M\rangle =L$. This would mean that $A$ and $M$ are adjacent, a contradiction. Hence $A \subseteq M$ for every maximal subalgebra $M$ of $L$, and thus $A \subseteq F(L)$.

Conversely, assume that $A \subseteq F(L)$, and let $B$ be any vertex of $\Gamma(L)$. If we choose a maximal subalgebra $M$ of $L$ containing $B$, we have $\langle A, B\rangle \subseteq M \neq L$. This shows that $A$ is not adjacent to any vertex of $\Gamma(L)$.
\end{proof}

Following the same arguments as in \cite[Theorem 3.1]{Das1}, we obtain the following result. 

\begin{theorem}
$\Gamma(L)$ is complete if and only if every proper subalgebra is one-dimensional.
\end{theorem}

\begin{proof}
Suppose $\Gamma(L)$ is complete. If $S$ is a proper subalgebra of $L$ and $\dim(S)> 1$,  pick any two linearly independent elements $x,y\in S$. Then $\langle \mathbb{F}x,\mathbb{F}y\rangle\subseteq S\neq L$, so $A=\mathbb{F}x$ and $B=\mathbb{F}y$ are not adjacent, contradicting completeness of $\Gamma(L)$.

Conversely, if every proper subalgebra of $L$ is one-dimensional, the subalgebra generated by any two of them is equal to $L$, and so $\Gamma(L)$ is complete.
\end{proof}

\smallskip

Recall that Lie algebras in which every subalgebra is one-dimensional were called $\mu$-algebras. The existence of such algebras of dimension greater than three is an interesting open problem. However, we have the following result.

\begin{theorem}\label{p:gen} 
Let $L$ be a Lie algebra over a perfect field $\mathbb{F}$ of characteristic zero or $p>3$. Then $L$ is a $\mu$-algebra if and only if either
\begin{itemize}
\item[(i)] $\dim L\leq 2$, or
\item[(ii)] $L$ is three-dimensional simple and $\sqrt{\mathbb{F}} \not \subseteq \mathbb{F}$. 
\end{itemize}
\end{theorem}

\begin{proof} 
Suppose first that $L$ is not simple. If $\dim L\geq 2$ there is a one-dimensional ideal $A$ of $L$ and $x\notin A$. But then $L=\langle A,x\rangle=A+\mathbb{F}x$ and $\dim L=2$. Hence, case (i) holds.
\par

So, assume now that $L$ is simple. Then $L$ has rank one and $\mathbb{F}x$ is a Cartan subalgebra of $L$. Let $\Gamma$ denote the centroid of $L$. Since $\Gamma x$ is an abelian subalgebra of $L$, we have that $\Gamma x < C_L(x) = \mathbb{F}x$. So $\Gamma = \mathbb{F}$, and $L$ is central-simple. Suppose that $\dim L > 3$. It follows from \cite{bo} that $L$ is a form of an Albert-Zassenhaus algebra. Moreover, $L$ has the one-and-a-half generation property. For, given any $y \in L$, either $y=\alpha x$ for some $\alpha \in \mathbb{F}$, in which case $\langle y, z \rangle = L$ for any $z \not \in \mathbb{F}x$, or else $y \not \in \mathbb{F}x$, and then $\langle y,x \rangle = L$. thus, $L$ is a form of a Zassenhaus algebra, by \cite{bois}.
\par

Let $\mathbb{K}$ be a splitting field for the minimal polynomial of ad\,$x$ over $\mathbb{F}$, and let $G$ be the Galois group of $\mathbb{K}$ over $\mathbb{F}$. Let $\sigma \in G$. Then $\sigma' = 1 \otimes \sigma$ is a Lie automorphism of $L \otimes_\mathbb{F} \mathbb{K} = L_\mathbb{K}$. As $\mathbb{K}$ is a Galois extension of $\mathbb{F}$, an element of $L_\mathbb{K}$ lies in $L$ if and only if it is fixed by $\sigma'$ for every
$\sigma \in G$. Now $L_\mathbb{K}$ has a unique maximal subalgebra $M$, containing $\mathbb{K}x$ of codimension one in $L_\mathbb{K}$, and $\sigma'$ must fix $M$. It follows that $(M\cap L)_\mathbb{K}=M$ (see \cite[p. 54]{borel}) and so $M\cap L$ is a subalgebra of $L$ of codimension one in $L$. We must have $M\cap L=\mathbb{F}x$, which is impossible. Hence $L$ is three-dimensional simple and, as is well known, has only one-dimensional maximal subalgebras if and only if $\sqrt{\mathbb{F}} \not \subseteq \mathbb{F}$; that is, $L$ is as in (ii).
\end{proof}

\begin{proposition}\label{diam}
Let $L$ be a finite-dimensional Lie algebra over a field $\mathbb{F}$, with $\dim L\geq 2$, and let $\Gamma(L)$ be its comaximal graph. Then the induced subgraph $\Gamma(L) \setminus F(L)$ is connected and has diameter at most $3$. In particular, if $\,F(L) = 0$, then $\Gamma(L)$ is connected and has diameter at most $3$.    
\end{proposition}

\begin{proof}
Let $A,B$ be non-adjacent subalgebras of $L$. If $A,B\not \subseteq F(L)$ there must be maximal subalgebras $M, K$ such that $A\not \subseteq M$, $B\not \subseteq K$. If $B\not \subseteq M$ we have $A\sim M\sim B$. If $B\subseteq M$, then $M \neq K$ and $A\sim M\sim K\sim B$.
\end{proof}
\smallskip

Note that the upper bound can be attained, as is shown by the following example.

\begin{example} 
Let $L= (\mathbb{F}a\oplus B)\dot{+}\mathbb{F}x$ where $[a,x]=0$, $\dim B =2$, $B$ is abelian and $ad\,x$ acts irreducibly on $B$. Put $A=\mathbb{F}a$. 
Then $F(L)=0$ and there is no path of length 2 from $A$ to $B$.
\par

For, let $A\sim M\sim  B$ be such a path. Then $L=\langle A, M\rangle =A+M$, so $M$ is a maximal subalgebra of codimension one in $L$. Also $L=\langle B,M\rangle=B+M$ and $B\cap M$ is an ideal of $L$. Hence $B\cap M=0$ and $M$ has codimension two, a contradiction.
\end{example}

Of course, the above example is only valid if the field is not algebraically closed, which leaves open the following question. Can the upper bound in Proposition \ref{diam} be improved to two over such a field?

Before proceeding to study the comaximal graphs for low-dimensional Lie algebras over a field $\mathbb{F}$, we record two general adjacency facts that hold for any three-dimensional Lie algebra over any field $\mathbb{F}$, and that will be used repeatedly throughout the next Section without further comment. The key observation is that in dimension three, the condition $\langle A, B \rangle = L$ depends only on the dimensions of $A$ and $B$, independently of the Lie bracket structure of $L$.

\begin{lemma}\label{g-adjacency}
Let $L$ be a three-dimensional Lie algebra over a field $\mathbb{F}$,  and let $\mathcal{L}$ denote the set of lines (one-dimensional Lie subalgebras) and $\mathcal{P}$ the set of planes (two-dimensional Lie subalgebras) of $L$.
\begin{enumerate}[(1)]
\item If $A, B \in \mathcal{P}$ with $A \neq B$, then $A \sim B$ in $\Gamma(L)$. In particular, the induced subgraph 
$\Gamma(L)[\mathcal{P}]$ is always a complete graph.

\item If $A \in \mathcal{L}$ and $B \in \mathcal{P}$, then $A \sim B$ in $\Gamma(L)$ if and only if $A \not\subseteq B$.
\end{enumerate}

\begin{proof}
\begin{enumerate}[(1)]
\item Since $A$ and $B$ are distinct two-dimensional subspaces of the three-dimensional space $L$, we have $\dim(A + B) = 3$, hence $\langle A, B \rangle = A + B = L$, and therefore $A \sim B$.
    
\item Let $A\in\mathcal{L}$ and $B\in\mathcal{P}$. If $A\subseteq B$, then $\langle A,B\rangle =B\neq L$, so $A$ and $B$ are not adjacent. Conversely, if $A\not\subseteq B$, then $\langle A,B\rangle = L$, and $A$ and $B$ are adjacent.
\end{enumerate}
\end{proof}
\end{lemma}

\section{The comaximal graphs for low-dimensional Lie algebras over finite fields}

\subsection{Comaximal graph of a 1-dimensional Lie algebra}
Let $L$ be a $1$-dimensional Lie algebra over the field $\F_q$. Then the comaximal graph $\Gamma(L)$ is the empty graph.

\subsection{Comaximal graph of a 2-dimensional Lie algebra}
Let $L$ be a $2$-dimensional Lie algebra over the field $\F_q$. Up to isomorphism, there are exactly 2 Lie algebras of dimension 2 over $\mathbb{F}_q$: the abelian one, and the non-abelian one with bracket $[x, y] = y$ (see \cite{humph}, Chapter 1). In both cases, every one-dimensional subspace of $L$ is a Lie subalgebra. Then the comaximal graph $\Gamma(L)$ is the complete graph $K_{q+1}$. In fact, the set of vertices $\mathcal{V}$ has $q+1$ elements. If $A, B\in \mathcal{V}$, with $A\neq B$, then $\langle A, B\rangle =A+B=L$. That is, every pair of distinct vertices is adjacent in the comaximal graph, and so  $\Gamma(L)$ is the complete graph $K_{q+1}$.

\subsection{Comaximal graph of a 3-dimensional Lie algebra}
Throughout this subsection, $L$ denotes a three-dimensional Lie algebra over the finite field $\mathbb{F}_q$. A unifying feature of the cases considered below is that, whenever $\dim L' \geq 1$, the algebra $L$ admits a decomposition $L = \langle x \rangle \ltimes V$, where $V$ is a two-dimensional ideal. The structure of $L$ is then determined by the linear transformation
\[T = \ad\, x\big|_V \in \mathrm{End}(V).\]
In this framework, the geometry of the comaximal graph $\Gamma(L)$ is governed by how $T$ acts on the one-dimensional subspaces of $V$. In particular, the distinction between $T$-invariant and non-invariant lines in $V$ plays a central role in determining adjacency between vertices. The different cases correspond to different types of the operator $T$:
\begin{itemize}
    \item In the Heisenberg case (Case 2A), $T$ is nilpotent;
    \item In the solvable non-nilpotent case with $\dim L' = 1$ (Case 2B), the image of $T$ is one-dimensional;
    \item In the case $\dim L' = 2$ (Case 3), the operator $T$ is invertible, and its Jordan type over $\mathbb{F}_q$ determines the number of two-dimensional subalgebras.
\end{itemize}
This notation and decomposition will be used without further comment throughout all cases below.

\begin{remark}\label{lines-count}
Using the notation introduced in Lemma \ref{g-adjacency}, since every one-dimensional subspace of a Lie algebra is automatically a Lie subalgebra, the set $\mathcal{L}$ coincides with the set of all one-dimensional subspaces of $L$, and therefore $|\mathcal{L}| = \tfrac{q^3 - 1}{q - 1} = q^2 + q + 1$. This fact holds independently of the Lie bracket structure of $L$, and will be used without further comment in all cases below. The number $|\mathcal{P}|$, by contrast, depends on the structure of $L$ and will be determined case by case.
\end{remark}


\textbf{Case 1. $\dim L'=0$ - The abelian algebra.} In this case, every subspace of $L$ is a Lie subalgebra. Hence, the proper nonzero Lie subalgebras of $L$ are all the 1-dimensional and 2-dimensional vector subspaces of the vector space $L$. The number of lines and planes in $L$ are respectively $q^2+q+1$.  

\begin{proposition}
The comaximal graph $\Gamma(L)$ has vertex set $\mathcal{V} = \mathcal{L} \sqcup \mathcal{P}$, and adjacency is described as follows:
\begin{enumerate}
\item If $A, B\in \mathcal{P}$ with $A\neq B$, then they are adjacent. Thus, the induced subgraph on $\mathcal{P}$ is the complete graph $K_{q^2+q+1}$.
\item If $A, B\in \mathcal{L}$, then they are never adjacent. Thus, the induced subgraph on $\mathcal{L}$ is an independent set.
\item For $A\in \mathcal{L}$ and $B\in \mathcal{P}$ one has $A\sim B$ if and only if $A\not\subseteq B$.
\end{enumerate}
In particular, $\deg(B)=2q^2+q$ for every $B\in \mathcal{P}$ and $\deg(A)=q^2$ for every $A\in \mathcal{L}$.
\end{proposition}

\begin{proof}
Clearly the vertex set $\mathcal{V}$ of $\Gamma(L)$ is $\mathcal{L} \sqcup \mathcal{P}$.

\smallskip

\noindent (1) By Lemma \ref{g-adjacency} (1), the induced subgraph on $\mathcal{P}$ is complete.


\smallskip

\noindent (2) Let $A, B\in \mathcal{L}$ with $A\neq B$. Then $\dim(A+B)=2$. Since $L$ is abelian, $\langle A,B\rangle =A+B$, so $\langle A,B\rangle \neq L$. Hence, $A$ and $B$ are not adjacent. Since this holds for every pair of distinct lines, the induced subgraph on $\mathcal{L}$ is an independent set.

\smallskip

\noindent (3) This follows immediately from Lemma~\ref{g-adjacency} (2).


\smallskip

Finally, note that a $2$-dimensional subspace $B$ contains exactly $q+1$ distinct lines, so there are $(q^2+q+1)-(q+1)=q^2$ lines not contained in $B$, each adjacent to $B$. Moreover, there are $q^2+q$ other planes, all adjacent to $B$ by (1). Hence $\deg(B) = q^2 + (q^2+q) = 2q^2 + q$.
Similarly, a $1$-dimensional subspace $A$ is contained in exactly $q+1$ planes, so there are $(q^2+q+1)-(q+1)=q^2$ planes not containing $A$, all adjacent to $A$ by (3), and no adjacent lines by (2). Therefore $\deg(A)=q^2$.
\end{proof}

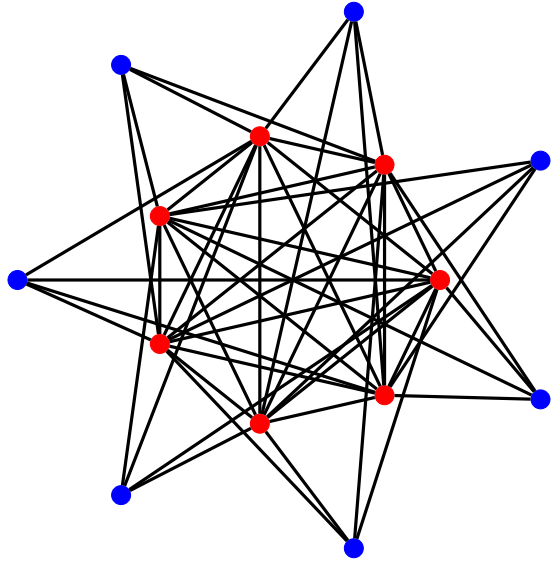
\begin{figure}[H]
\centering
\begin{tikzpicture}[
    scale=1.3,
    plane/.style={circle, draw=red, fill=red, minimum size=2.5mm, inner sep=2pt},
    linev/.style={circle, draw=blue, fill=blue, minimum size=2.5mm, inner sep=2pt},
    every edge/.style={draw=black}
]

\def\rinner{1.5}   
\def\router{2.8}     

\foreach \i in {0,...,6}{
  \pgfmathsetmacro{\ang}{360/7*\i}
  \node[plane] (P\i) at ({\rinner*cos(\ang)}, {\rinner*sin(\ang)}) {};
}

\foreach \i in {0,...,6}{
  \pgfmathsetmacro{\ang}{360/7*\i + 360/14}
  \node[linev] (L\i) at ({\router*cos(\ang)}, {\router*sin(\ang)}) {};
}

\foreach \i in {0,...,6}{
  \foreach \j in {0,...,6}{
    \ifnum\i<\j
      \draw[very thick] (P\i)--(P\j);
    \fi
  }
}

\foreach \p/\a/\b/\c/\d in {
  0/3/4/5/6,
  1/1/2/5/6,
  2/1/2/3/4,
  3/0/2/4/6,
  4/0/2/3/5,
  5/0/1/4/5,
  6/0/1/3/6}
{
  \foreach \l in {\a,\b,\c,\d}{
    \draw[very thick] (P\p)--(L\l);
  }
}
\end{tikzpicture}

\qquad

\caption{The comaximal graph of a 3-dimensional abelian Lie algebra over $\mathbb{F}_2$.}
\label{fig:placeholder}
\end{figure}

\textbf{Case 2A. $\dim L'=1$ - The nonabelian nilpotent algebra (The Heisenberg algebra).}
Let $B= (e,f,h)$ be a basis for $L$, with brackets $[e,f]=h$, and $[e,h]=[f,h]=0$. Let $Z:= Z(L) = L' = \langle h\rangle$ be its center and derived algebra. Note that in this case, the derived algebra is central, so the bracket may produce a new independent direction, which explains the existence of edges between lines.

\begin{proposition}\label{heisenberg}
The line $Z$ is central, and the other $q^2+q$ lines are noncentral. Further, the set $\mathcal{P}$ consists of the $q+1$ planes containing $Z$. The comaximal graph $\Gamma(L)$ has vertex set $V = \mathcal{L} \sqcup \mathcal{P}$, and adjacency is given by:
\begin{enumerate}[(1)]
\item If $A, B \in \mathcal{P}$ with $A \neq B$, then $A \sim B$. Hence $\Gamma(L)[\mathcal{P}]\cong K_{q+1}$. That is, the induced subgraph on $\mathcal{P}$ is the complete graph $K_{q+1}$.

\item The central line $Z$ is an isolated point in $\Gamma(L)$, and two noncentral lines $A, B \in \mathcal{L}$ are adjacent if and only if they are not both contained in the same plane of $\mathcal{P}$. Consequently, the induced subgraph on $\mathcal{L} \setminus \{Z\}$ is the complete multipartite graph $K_{\underbrace{q,\ldots, q}_{q+1}}$, with $q+1$ parts of size $q$.

\item For $A \in \mathcal{L}$ and $B \in \mathcal{P}$, $A \sim B$ if and only if $A \not\subseteq B$. In particular, $Z$ is an isolated vertex.
\end{enumerate}
Further, every noncentral line has degree $q^2 + q$, and every plane has degree $q^2 + q$.
\end{proposition}

\begin{proof}
By Remark \ref{lines-count}, $|\mathcal{L}| = q^2 + q + 1$. The center $Z = \langle h \rangle$ is one of these lines; the others are noncentral.

We prove now that the $2$-dimensional Lie subalgebras of $L$ are precisely the planes containing $h$. In fact, let $U$ be a $2$-dimensional subspace of $L$. If $h \in U$ then $[U,U] \subseteq \F_q h \subseteq U$, so $U$ is a subalgebra. 
If $h \notin U$, since $\dim L/\langle h\rangle =2$, the image of $U$ in $L/\langle h\rangle$ under the canonical projection $\pi$ is all of $L/\langle h\rangle$. Choose $u_1=a_1e+b_1f+c_1h$, $u_2=a_2e+b_2f+c_2h \in U$ such that $\{\pi(u_1), \pi(u_2)\}$ form a basis of $L/\langle h\rangle$. The linear independence of $\{\pi(u_1), \pi(u_2)\}$ in $L/\langle h\rangle$ means that $a_1b_2-a_2b_1\neq 0$. Hence
\[[u_1,u_2]=(a_1b_2-a_2b_1)h\neq 0.\]
Thus $[u_1,u_2]$ is a nonzero element of $\langle h\rangle$. Since $h\notin U$, we have $\langle h\rangle\cap U= 0$, and therefore $[u_1,u_2]\notin U$. It follows that $U$ is not closed under the bracket, and therefore $U$ is not a Lie subalgebra. This proves that the $2$-dimensional Lie subalgebras are precisely the planes containing $Z$, and their number is $q+1$.

\qquad

It is clear that $\mathcal{V} = \mathcal{L}\sqcup \mathcal{P}$. We now verify the adjacency rules.

\noindent \emph{(1)} Let $A, B \in \mathcal{P}$ with $A \neq B$ containing $Z$, then $A\cap B = Z$, and the rest follows from Lemma~\ref{g-adjacency} (1).


\smallskip

\noindent  \emph{(2)} Let $A = \langle a \rangle$ and $B = \langle b \rangle$ be two distinct lines in $\mathcal{L}$, with $a = a_1e+b_1f+c_1h$ and $b = a_2 e + b_2 f + c_3 h$. As above we have $[a,b] = (a_1 b_2 - a_2b_1) h$. If $[a,b]\neq 0$, then $h\in \langle A,B\rangle$, and therefore $\langle A, B\rangle =L$ and so $A \sim B$. If $[a,b] = 0$, then $A+B$ is an abelian $2$-dimensional subalgebra, so $\langle A, B\rangle = A+B \neq L$. Hence, $A \not\sim B$.

Furthermore, the condition $a_1 b_2 - a_2b_1  = 0$ means that the vectors $(a_1, b_1)$ and $(a_2, b_2)$ are proportional in $\F_q^2$, that is, $a$ and $b$ have proportional projections onto $\langle e, f \rangle$ modulo $\langle h \rangle$.
This holds if and only if $a$ and $b$ lie in the same plane of $\mathcal{P}$, that is, $A$ and $B$ are both contained in some $P \in \mathcal{P}$. In particular, if $A = Z = \langle h \rangle$, then $a_1 = b_1 = 0$, so $a_1 b_2 - a_2b_1 = 0$ for every choice of $B$. Hence, $Z$ is not adjacent to any line, confirming that $Z$ is an isolated point.

For noncentral lines, the $q^2 + q$ elements of $\mathcal{L} \setminus \{Z\}$ are partitioned into $q+1$ groups of size $q$ according to which plane of $\mathcal{P}$ contains them. Two noncentral lines are adjacent if and only if they belong to different groups. Therefore, the induced subgraph on $\mathcal{L} \setminus \{Z\}$ is the complete multipartite graph $K_{\underbrace{q,\ldots,q}_{q+1}}$.

\smallskip

\noindent  \emph{(3)} This follows immediately from Lemma~\ref{g-adjacency} (2).

Finally, let $B \in \mathcal{P}$. By (1), $B$ is adjacent to the other $q$ planes. By (3), $B$ is adjacent to the
lines not contained in $B$. Each plane contains exactly $q+1$ lines (the central line $Z$ and $q$ noncentral lines), so the number of lines not in $B$ is $(q^2+q+1) - (q+1) = q^2$. Hence $\deg(B) = q + q^2= q^2+q$.

The central line $Z$ is contained in every plane of $\mathcal{P}$ and is not adjacent to any line, so $\deg(Z) = 0$.
Let $A =\langle a \rangle$, with $a = a_1e+b_1f+c_1h$, be a noncentral line. Since $A$ is noncentral, it is contained in exactly one plane, namely $\langle A, h \rangle \in \mathcal{P}$. As $|\mathcal{P}|= q+1$, by (3)  $A$ is adjacent to the $q$ planes not containing it.

On the other hand, let  $A$ be a noncentral line. Then $A$ is contained in a unique plane, namely, $P_A = \langle A, h\rangle\in \mathcal{P}$. The lines not adjacent to $A$ are precisely the lines contained in $P_A$, excluding $A$ itself. Since a plane contains exactly $q+1$ lines, this yields exactly $q$ lines not adjacent to $A$. As $|\mathcal{L}\setminus A| = q^q+q$, it follows that $A$ is adjacent to $q^2$ lines. Moreover, $A$ is adjacent to the $q$ planes that do not contain it. Therefore, $\deg(A) = q^2+q$.
\end{proof}

\begin{figure}[H]
\centering
\begin{tikzpicture}[
    scale=1.3,
    plane/.style={circle, draw=red, fill=red, minimum size=2.5mm, inner sep=2pt},
    linev/.style={circle, draw=blue, fill=blue, minimum size=2.5mm, inner sep=2pt},
    central/.style={circle, draw=black, fill=black, minimum size=2.5mm, inner sep=2pt},
    every edge/.style={draw=black}
]

\node[plane] (P1) at (90:3) {};
\node[plane] (P2) at (210:3) {};
\node[plane] (P3) at (330:3) {};

\node[linev] (A1) at (90:2) {};
\node[linev] (A2) at (90:1) {};

\node[linev] (B1) at (210:2) {};
\node[linev] (B2) at (210:1) {};

\node[linev] (C1) at (330:2) {};
\node[linev] (C2) at (330:1) {};

\node[central] (Z) at (0,0) {};

\draw[very thick] (P1)--(P2)--(P3)--(P1);

\foreach \a in {A1,A2}{
  \foreach \b in {B1,B2}{
    \draw[very thick] (\a)--(\b);
  }
  \foreach \c in {C1,C2}{
    \draw[very thick] (\a)--(\c);
  }
}
\foreach \b in {B1,B2}{
  \foreach \c in {C1,C2}{
    \draw[very thick] (\b)--(\c);
  }
}

\foreach \a in {A1,A2}{
  \draw[very thick] (\a)--(P2);
  \draw[very thick] (\a)--(P3);
}
\foreach \b in {B1,B2}{
  \draw[very thick] (\b)--(P1);
  \draw[very thick] (\b)--(P3);
}
\foreach \c in {C1,C2}{
  \draw[very thick] (\c)--(P1);
  \draw[very thick] (\c)--(P2);
}
\end{tikzpicture}

\qquad

\caption{The comaximal graph $\Gamma(H_3(\mathbb F_2))$. Blue vertices are noncentral lines, red vertices are planes, and the isolated black vertex is the central line $Z=\langle h\rangle$.}
\end{figure}
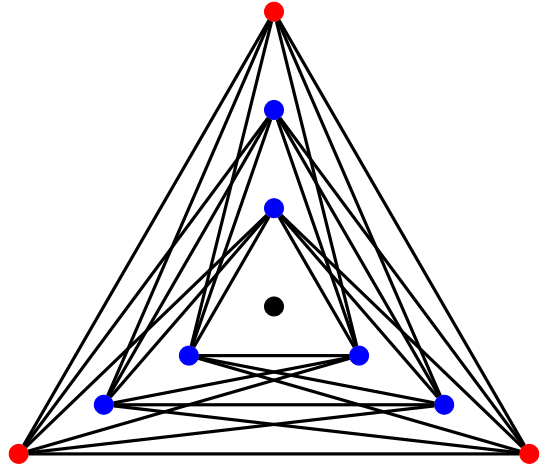

\begin{corollary}\label{regularity}
The graph $\Gamma(L)^*$ is $(q^2+q)$-regular of order $(q+1)^2$.
\end{corollary}

\begin{proof}
By Proposition~\ref{heisenberg}, every vertex in $V \setminus \{Z\}$ has degree $q^2+q$. The order of $V \setminus \{Z\}$ is $(q^2+q+1) -1 + (q+1) = (q+1)^2$.
\end{proof}

\textbf{Case 2B. $\dim L'=1$ - A solvable non-nilpotent algebra.}
In this case, $L' \neq Z(L)$. There exists a basis $\{x,y,z\}$ of $L$ such that the Lie bracket is given by $[x,y] = y$, and $[x,z] = [y,z] = 0$. In this situation we have $L' = \langle y \rangle$, and $Z(L) = \langle z \rangle$. 

Unlike the Heisenberg case, the line $L'=\langle y\rangle$ is not central and is not contained in every plane.
Therefore, $\Gamma(L)$ has no isolated vertex arising from centrality. Moreover, the number of planes differs significantly from both the abelian and Heisenberg cases, reflecting the non-nilpotent solvable structure of $L$. Here, the adjoint action satisfies $L' \subseteq \langle y \rangle$, so brackets remain confined to a proper subspace, and adjacency between lines is more restrictive.

\begin{proposition}\label{dim3-derived1-noncentral}
The elements of $\mathcal{P}$ are of the following two types:
\begin{enumerate}[(i)]
\item  the $q+1$ planes containing $L'=\langle y\rangle$;
\item  the $q$ abelian planes $P_\alpha := \langle x+\alpha y,\; z\rangle$, with $\alpha\in \mathbb{F}_q$, which do not contain $y$. In particular, $|\mathcal{P}| = 2q+1$.
\end{enumerate}

The comaximal graph $\Gamma(L)$ has vertex set $V(\Gamma(L))=\mathcal{L}\sqcup \mathcal{P}$, and adjacency is described as follows:
\begin{enumerate}[(1)]
\item If $A, B\in \mathcal{P}$ with $A\neq B$, then $A\sim B$. Hence, the induced subgraph on $\mathcal{P}$ is the complete graph $K_{2q+1}$.
\item  If $A\in \mathcal{L}$ and $P\in \mathcal{P}$, then $A\sim P$ if and only if $A\not\subseteq P$.

\item Let $V := \langle y, z\rangle$. Then the adjacency between lines is described as follows:

\begin{enumerate}
\item The lines $\langle y\rangle$ and $\langle z\rangle$ are not adjacent to any line.

\item If $A \leq V$ with $\langle y\rangle \neq A \neq\langle z\rangle$, then $A$ is adjacent to every line not contained in $V$, and to no line contained in $V$.

\item Every line not contained in $V$ can be written uniquely as $X_{a,b} = \langle x+a y+b z\rangle$, with $a, b\in\F_q$. 
For two such lines $X_{a,b}$ and $X_{c,d}$ one has $X_{a,b}\sim X_{c,d}$ if and only if $a\neq c$ and $b\neq d$.
\end{enumerate}
\item The degrees of the vertices are as follows:
\begin{enumerate}
\item $\deg(P) = q^2 + 2q$, for every $P \in \mathcal{P}$. 
\item $\deg(\langle y\rangle) = \deg(\langle z\rangle) = q$.
\item If $A \leq V$ with $\langle y\rangle \neq A \neq\langle z\rangle$, then $\deg(A) = q^2 + 2q$.
\item If $A \not\leq V$, then $\deg(A) = q^2 + q - 1$.
\end{enumerate}
\end{enumerate}
\end{proposition}

\begin{proof}
By Remark \ref{lines-count}, $|\mathcal{L}| = q^2 + q + 1$. In the notation of the introduction of this subsection, we have $V = \langle y, z \rangle$, which is an abelian ideal of $L$ and its 
nilradical. The operator $T = \ad x|_V$ satisfies $\mathrm{Im}(T) = \langle y \rangle$, so $L' = \langle y \rangle \subseteq V$.


We first determine the two-dimensional subalgebras of $L$. Let $U$ be a $2$-dimensional subspace of $L$.
If $y\in U$, then $[U,U]\subseteq \langle y\rangle \subseteq U$, so $U$ is a Lie subalgebra. There are $q+1$ such planes.
If $y\notin U$, then $U$ is a Lie subalgebra if and only if it is abelian. Since $U\not\subseteq V$, it contains an element of the form $x+\alpha y+\beta z$, with $\alpha, \beta\in \F_q$. Replacing this element by $x+\alpha y$, we may assume $x+\alpha y \in U$. 

If $u\in U$, then $U$ is abelian if and only if $[x+\alpha y, u]=0$. Writing $u=ax+by+cz$, this condition gives $b=\alpha a$, hence $u \in \langle x+\alpha y, z\rangle$. Thus $U=\langle x+\alpha y, z\rangle =: P_\alpha$. Therefore the $2$-dimensional subalgebras are exactly the $q+1$ planes containing $\langle y\rangle$ together with the $q$ planes $P_\alpha$, and so $|\mathcal{P}|=2q+1$.

\noindent (1) By Lemma \ref{g-adjacency} (1), the induced subgraph on $\mathcal{P}$ is complete.

\noindent (2) This follows immediately from Lemma~\ref{g-adjacency} (2).

\noindent (3) We now describe adjacency between lines.

The lines $\langle y\rangle$ and $\langle z\rangle$ are not adjacent to any line: $\langle z\rangle$ is central, and $\langle y\rangle$ together with any line generates a subalgebra contained in a proper $2$-dimensional subspace.

If $A, B\leq V$, then $\langle A, B\rangle \leq V \neq L$, so they are not adjacent.

Let $A\leq V$ with $\langle y\rangle \neq A \neq\langle z\rangle$. Then $A=\langle y+t z\rangle$ for some $t\neq 0$. If $B\not\leq V$, then $B$ contains a nonzero $x$-component, and $[y+t z, B] \subseteq \langle y\rangle \setminus \{0\}$.
Hence $y\in \langle A,B\rangle$, and since $A$ also contains a nonzero $z$-component, we obtain $z$. Together with the $x$-component of $B$, this yields $\langle A,B\rangle=L$. Thus, such $A$ is adjacent to every line not contained in $V$.

Now let $A= X_{a,b}$ and $B= X_{c,d}$ be lines not contained in $V$. Then $[A,B]=(c-a)y$.

If $a=c$, then $\langle A,B\rangle=A+B$ is $2$-dimensional, so $A\not\sim B$.

If $a\neq c$, then $y\in \langle A,B\rangle$. Subtracting generators gives $B-A=(c-a)y+(d-b)z$.
Thus $z\in \langle A,B\rangle$ if and only if $b\neq d$. If both $y$ and $z$ lie in $\langle A, B\rangle$, then $x$ is recovered from either generator, and hence $\langle A, B\rangle=L$. Therefore, $A\sim B$ if and only if $a\neq c$ and $b\neq d$.

\noindent (4) Finally, we compute degrees. Each plane is adjacent to the other $2q$ planes and to the $q^2$ lines not contained in it, so $\deg(P)=q^2+2q$.

The lines $\langle y\rangle$ and $\langle z\rangle$ are adjacent only to the $q$ planes not containing them, so their degree is $q$.

If $A\leq V$ and $A\neq \langle y\rangle,\langle z\rangle$, then $A$ is adjacent to all $q^2$ lines outside $V$ and to $2q$ planes, so $\deg(A)=q^2+2q$.

If $A\not\leq V$, then $A$ is adjacent to $(q-1)^2$ lines outside $V$ and to $q-1$ lines in $V$, giving $q(q-1)$ adjacent lines, and to $2q-1$ planes, so $\deg(A)=q^2+q-1$.
\end{proof}

\smallskip

\textbf{Case 3. $\dim L'=2$ - Solvable non-nilpotent algebras.}
Here, $L'$ is a 2-dimensional abelian ideal of $L$. In the notation of the introduction of this subsection, the structure of $\Gamma(L)$ depends on the Jordan type of the operator $T = \ad x |_V$ over $\mathbb{F}_q$, which determines the number and type of two-dimensional subalgebras of $L$. We distinguish four subcases according to the canonical form of $T$ over $\mathbb{F}_q$.

\begin{lemma} \label{2dim-subalgebras}
Let $L=Fx+V$ where $V=Fv_1+Fv_2$ and $[x,v_1]=\alpha v_1+\beta v_2$, $[x,v_2]=\gamma v_1+\delta v_2$ where $A=\left(\begin{matrix} \alpha&\beta\\\gamma&\delta \end{matrix}\right)$ is a non-singular matrix. Then
\begin{itemize}
\item[(i)] If $\ad x$ has no eigenvector in $V$, the only two-dimensional subalgebra of $L$ is $V$; 
\item[(ii)] if $\ad x$ has an eigenvector in $V$, the two-dimensional subalgebras of $L$ are $V$ or of the form $Fx+Fv$ where $v$ is an eigenvector, and $Fv+F(x+v_3)$ where $v$ is an eigenvector and $v_3$ is linearly independent of $v$.
\end{itemize}
\end{lemma}

\begin{proof} 
Let $U$ be a two-dimensional subalgebra of $L$. If $x\in U$, then $U=Fx+Fv$ where $v$ is an eigenvector of $\ad x$. So suppose that $x\notin U$. Then $L=Fx+U$ and $U=F(\lambda x+v_1)+F(\mu x+v_2)$. If $\lambda=\mu=0$, then $U=V$, so suppose this is not the case. Then $v=\mu v_1-\lambda v_2\in U$. Let $v_3$ be linearly independent of $v$, so that $x+\nu v_3\in U$. Now $[x,v]=\theta v +\phi(x+\nu v_3)$, so $\phi=0$ and $v$ is an eigenvector of $\ad x$ in $V$.
\end{proof}

\smallskip

So there are four cases for algebras in this case.
\smallskip

\noindent {\bf Case 1: The characteristic equation of $A$ is irreducible.} Then the only two-dimensional subalgebra is $V$.  
\smallskip

\noindent {\bf Case 2: There are two different eigenvalues $\lambda$ and $\mu$.} Then, replacing $x$ by $\lambda^{-1}x$, we have a basis such that $[x,v_1]=v_1$, $[x,v_2]=\mu v_2$. The two-dimensional subalgebras are $V$, $Fv_1+F(x+\alpha v_2)$ and $Fv_2+F(x+\alpha v_1)$.  

\smallskip

\noindent {\bf Case 3: There is only one eigenvalue $\lambda$ and $A$ is not diagonable.} Then there is a basis such that $[x,v_1]=\lambda v_1$, $[x,v_2]=v_1+\lambda v_2$. The two-dimensional subalgebras are $V$, and $Fv_1+F(x+\alpha v_2)$.  
\smallskip

\noindent {\bf Case 4: $A$ is a scalar matrix, $\lambda I_2$.}  Then, replacing $x$ by $\lambda^{-1} x$, there is a basis such that $[x,v_1]=v_1$, $[x,v_2]=v_2$.

\begin{proposition}\label{dim3-derived2}
Let $L$ be a three-dimensional Lie algebra over the finite field $\F_q$ such that $\dim L' = 2$. As before, let $\mathcal{L}$ denote the set of one-dimensional Lie subalgebras of $L$, and let $\mathcal{P}$ denote the set of two-dimensional Lie subalgebras of $L$. Then 
\begin{enumerate}
\item $\mathcal{L}$ has $q^2 + q + 1$ elements.
\item The number of two-dimensional Lie subalgebras is determined by the canonical form of $\ad x$ over $\F_q$ as follows:
\[|\mathcal{P}| = \begin{cases}
    1 &  \text{in case 1}\\
    1+2q &  \text{in case 2}\\
    1+q &  \text{in case 3}\\
    1+q+q^2 &  \text{in case 4}.
\end{cases}\]

\item The following rules hold in all four cases:
\begin{enumerate}
    \item If $A, B\in \mathcal{P}$ with $A\neq B$, then $A\sim B$. Hence, the induced subgraph on $\mathcal{P}$ is a complete graph.
    \item If $A \in \mathcal{L}$ and $P \in \mathcal{P}$, then $A \sim P$ if and only if $A \not\subseteq P$.
    \item If $A, B \in \mathcal{L}$ are distinct, then $A \not\sim B$ if and only if $A$ and $B$ are contained in a common plane of $\mathcal{P}$.
\end{enumerate}
    In particular:
\begin{enumerate}[(i)]
    \item In Case 1, $\Gamma(L)[\mathcal{L}]$ is the join of an independent 
    set on the $q+1$ lines of $V$ with a complete graph $K_{q^2}$ on the 
    lines outside $V$.
    \item In Case 3, the unique eigenline $\langle v_1\rangle$ is an isolated vertex of $\Gamma(L)$, since it is contained in every plane of $\mathcal{P}$.
    \item In Case 4, no two lines are adjacent, since every pair of lines 
    spans a plane in $\mathcal{P}$.
\end{enumerate}
\end{enumerate}
\end{proposition}

\begin{proof}
(1) Since every one-dimensional subspace of $L$ is a Lie subalgebra, it follows that $|\mathcal{L}| = q^2+q+1$.

\noindent (2) It follows immediately from Lemma \ref{2dim-subalgebras}. 

\noindent (3) Independently of the case, we have:  
\begin{enumerate}[(a)]
\item By Lemma \ref{g-adjacency} (1), the induced subgraph on $\mathcal{P}$ is complete.
    
\item This follows immediately from Lemma~\ref{g-adjacency} (2).

\item Finally, let $A,B\in\mathcal L$ be distinct. If $A$ and $B$ are contained in a common plane $P\in\mathcal P$, then $\langle A,B\rangle\subseteq P\neq L$, so $A\not\sim B$.

Conversely, suppose that $A\not\sim B$. Then $\langle A, B\rangle\neq L$. Since $A\neq B$, the space $A+B$ is two-dimensional, and hence $\langle A, B\rangle$ is a two-dimensional subalgebra of $L$. Therefore, $\langle A, B\rangle \in\mathcal P$, and both $A$ and $B$ are contained in
this plane. 
\end{enumerate}
The statements (i)-(iii) follow immediately.
\smallskip
\end{proof}

\begin{corollary}
Let $L$ be a Lie algebra as in Lemma \ref {2dim-subalgebras}. Then the degrees of the vertices of $\Gamma(L)$ are as follows.
\smallskip

\noindent\textbf{Case 1.} $|\mathcal{P}|=1$.
\begin{enumerate}[(a)]
    \item $\deg(V) = q^2$.
    \item $\deg(A) = q^2$ for every line $A \leq V$.
    \item $\deg(B) = q^2+q+1$ for every line $B \not\leq V$.
\end{enumerate}
In particular, every line outside $V$ is adjacent to every other vertex of $\Gamma(L)$.

\smallskip
\noindent\textbf{Case 2.} $|\mathcal{P}|=1+2q$. Let $E_1, E_2$ be the two $T$-eigenlines in $V$.
\begin{enumerate}[(a)]
    \item $\deg(P) = q^2+2q$ for every $P \in \mathcal{P}$.
    \item $\deg(E_1) = \deg(E_2) = q$.
    \item $\deg(A) = q^2+2q$ for every non-eigenline $A \leq V$.
    \item $\deg(B) = q^2+q-1$ for every line $B \not\leq V$.
\end{enumerate}
In particular, planes and non-eigenlines in $V$ have the same degree 
$q(q+2)$.

\smallskip
\noindent\textbf{Case 3.} $|\mathcal{P}|=1+q$. Let $E = \langle v_1\rangle$ be the unique $T$-eigenline in $V$.
\begin{enumerate}[(a)]
    \item $\deg(E) = 0$; that is, $E$ is an isolated vertex of $\Gamma(L)$.
    \item For every $P \in \mathcal{P}$ and every line $A \in \mathcal{L}$: $\deg(P) = \deg(A) =q^2+q$.
\end{enumerate}
In particular, $\Gamma(L)^*$ is a $(q^2+q)$-regular graph.

\smallskip
\noindent\textbf{Case 4.} $|\mathcal{P}|=q^2+q+1$.
\begin{enumerate}[(a)]
    \item $\deg(P) = 2q^2+q$ for every $P \in \mathcal{P}$.
    \item $\deg(A) = q^2$ for every line $A \in \mathcal{L}$.
\end{enumerate}
\end{corollary}

\begin{proof}
We use Proposition \ref{dim3-derived2} throughout. Since every two-dimensional subspace contains exactly $q+1$ lines, and there are $q^2+q+1$ lines, each plane $P \in \mathcal{P}$ is adjacent to exactly $q^2$ lines. We use this freely below.

\smallskip
\noindent\textbf{Case 1.}
(a) $\mathcal P=\{V\}$, so $V$ is adjacent only to the $q^2$ lines not contained in it. Hence $\deg(V) = q^2$.

(b) Let $A$ be a line contained in $V$. By Proposition \ref{dim3-derived2} (3)(c), since the only plane is $V$ and $A \leq V$, the only lines not adjacent to $A$ are those sharing the plane $V$ with $A$, namely the other $q$ lines in $V$. Hence $A$ is adjacent to $q^2$ lines outside $V$, giving $\deg(A) = q^2$.

(c) Now let $B$ be a line with $B\nleq V$. Since the only plane in $\mathcal P$ is $V$, no line outside $V$ is contained in a common plane with $B$. Thus $B$ is adjacent to every other line, namely to the $q+1$ lines of $V$ and to the remaining $q^2-1$ lines outside $V$. Also, $B\sim V$ by Proposition \ref{dim3-derived2} (3)(b). Hence $\deg(B)=(q+1)+(q^2-1)+1=q^2+q+1$. 

\smallskip

\noindent\textbf{Case 2.} (a) Here $|\mathcal P|=1+2q$. Let $P\in\mathcal P$. Since $\Gamma(L)[\mathcal P]$ is complete, $P$ is adjacent to the other $2q$ planes. Also, every plane contains exactly $q+1$ lines, so $P$ is adjacent to the remaining $(q^2+q+1)-(q+1)$ lines. Hence $\deg(P)=q^2+2q$.

(b) The eigenline $E_1= Fv_1$ is contained in $V$ and also in each of the $q$ planes of the form $Fv_1+F(x+\alpha v_2)$. Hence $E_1$ is contained in exactly $q+1$ planes. Moreover, every line distinct from $E_1$ is contained in $E_1$ in one of these planes, so by Proposition \ref{dim3-derived2} (3)(c) the line $E_1$ is not adjacent to any other line. Therefore, $E_1$ is adjacent only to the planes not containing it, of which there are $(1+2q)-(q+1)$. Thus, $\deg(E_1)=q$. The same argument applies to $E_2$.

(c) Now let $A\leq V$ be a non-eigenline. Then $A$ is contained only in the plane $V$, so by Proposition \ref{dim3-derived2} (3)(c) it is non-adjacent precisely to the other $q$ lines of $V$, and adjacent to all $q^2$ lines outside $V$. Also, by
Proposition \ref{dim3-derived2} (3)(b), it is adjacent to every plane except $V$, that is, to $2q$ planes. Hence, $\deg(A) = 2q+q^2 = q(q+2)$.

(d) Finally, let $A\nleq V$. Then $A$ is contained in exactly two planes of $\mathcal P$, one of the form $Fv_1+F(x+\alpha v_2)$ and one of the form $Fv_2+F(x+\beta v_1)$. Each of these planes contributes exactly $q$ lines
distinct from $A$, and since the two planes intersect exactly in $A$, these $2q$ lines are all distinct. Therefore, $A$ is non-adjacent to exactly $2q$ lines, and hence adjacent to $(q^2+q)-2q$ lines. It is also adjacent to all planes except the two containing it, that is, to $(1+2q)-2$ planes. Thus, $\deg(A)= (q^2+q)-2q + (1+2q)-2 = q^2+q-1$.
 
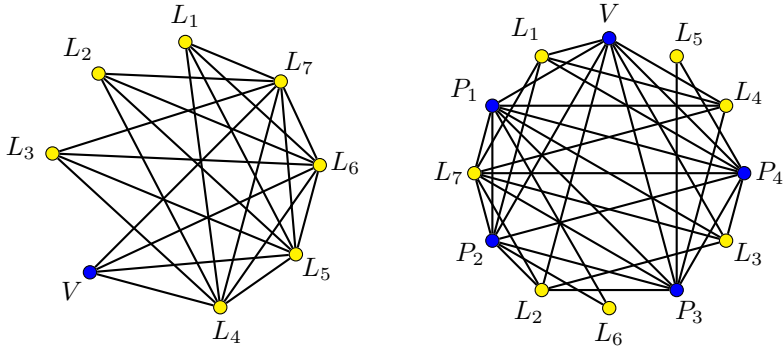
\begin{figure}[H]
\centering
\begin{minipage}{0.45\textwidth}
\centering
\begin{tikzpicture}[
  scale=0.85,
  linev/.style={circle,draw=black,fill=yellow, minimum size=5pt,inner sep=0pt},
  planev/.style={circle,draw=black,fill=blue, minimum size=5pt,inner sep=0pt}
]
\def\r{2.1}
\node[linev] (L1) at (90:\r)  {};
\node[font=\small] at (90:{\r+.4}) {$L_1$};
\node[linev] (L2) at (130:\r) {};
\node[font=\small] at (130:{\r+.5}) {$L_2$};
\node[linev] (L3) at (170:\r) {};
\node[font=\small] at (170:{\r+.5}) {$L_3$};


\node[planev] (Vp) at (225:\r) {};
\node[font=\small] at (225:{\r+.4}) {$V$};
\node[linev] (L4) at (285:\r) {};
\node[font=\small] at (285:{\r+.4}) {$L_4$};
\node[linev] (L5) at (325:\r) {};
\node[font=\small] at (325:{\r+.4}) {$L_5$};
\node[linev] (L6) at (5:\r)   {};
\node[font=\small] at (5:{\r+.4}) {$L_6$};
\node[linev] (L7) at (45:\r)  {};
\node[font=\small] at (45:{\r+.4}) {$L_7$};
\draw[thick] (L4)--(L5) (L4)--(L6)  (L4)--(L7) (L5)--(L6)  (L5)--(L7) (L6)--(L7);
\foreach \i in {1,2,3}{
  \foreach \j in {4,5,6,7}{\draw[thick] (L\i)--(L\j);}
}
\foreach \j in {4,5,6,7}{\draw[thick] (Vp)--(L\j);}
\end{tikzpicture}
\end{minipage}
\hfilneg
\begin{minipage}{0.50\textwidth}
\centering
\begin{tikzpicture}[
  scale=0.85,
  linev/.style={circle,draw=black,fill=yellow, minimum size=5pt,inner sep=0pt},
  planev/.style={circle,draw=black,fill=blue, minimum size=5pt,inner sep=0pt}
]
\def\r{2.1} 
\node[planev] (V)  at (90:\r)  {};
\node[font=\small] at (90:{\r+.4}) {$V$};
\node[planev] (Pa) at (150:\r) {};
\node[font=\small]  at (150:{\r+.5}) {$P_1$};
\node[planev] (Pb) at (210:\r) {};
\node[font=\small] at (210:{\r+.4}) {$P_2$};
\node[planev] (Pc) at (300:\r) {};
\node[font=\small] at (300:{\r+.4}) {$P_3$};
\node[planev] (Pd) at (0:\r)  {};
\node[font=\small]  at (0:{\r+.4}) {$P_4$};
\node[linev] (E1) at (60:\r)  {};
\node[font=\small] at (60:{\r+.4}) {$L_5$};
\node[linev] (La) at (120:\r) {};
\node[font=\small]  at (120:{\r+.5}) {$L_1$};
\node[linev] (Wn) at (180:\r) {};
\node[font=\small]  at (180:{\r+.4}) {$L_7$};
\node[linev] (Lb) at (240:\r) {};
\node[font=\small] at (240:{\r+.4}) {$L_2$};
\node[linev] (E2) at (270:\r) {};
\node[font=\small] at (270:{\r+.4}) {$L_6$};
\node[linev] (Lc) at (330:\r) {};
\node[font=\small]  at (330:{\r+.4}) {$L_3$};
\node[linev] (Ld) at (30:\r)  {};
\node[font=\small]  at (30:{\r+.4}) {$L_4$};
\draw[thick] (V)--(Pa)  (V)--(Pb)  (V)--(Pc)  (V)--(Pd);
\draw[thick] (Pa)--(Pb)  (Pa)--(Pc) (Pa)--(Pd) (Pb)--(Pc)  (Pb)--(Pd) (Pc)--(Pd);
\draw[thick] (E1)--(Pc)  (E1)--(Pd);
\draw[thick] (E2)--(Pa)  (E2)--(Pb);
\draw[thick] (Wn)--(Pa)  (Wn)--(Pb)  (Wn)--(Pc)  (Wn)--(Pd) (Wn)--(La)  (Wn)--(Lb) (Wn)--(Lc) (Wn)--(Ld);
\draw[thick] (La)--(V)  (La)--(Pb)  (La)--(Pd) (La)--(Ld);
\draw[thick] (Lb)--(V)  (Lb)--(Pb)  (Lb)--(Pc)  (Lb)--(Lc);
\draw[thick] (Lc)--(V)  (Lc)--(Pa)  (Lc)--(Pd);
\draw[thick] (Ld)--(V)  (Ld)--(Pa)  (Ld)--(Pc);
\end{tikzpicture}
\end{minipage}
\caption{The comaximal graphs for $q=2$ in cases 1 and 2}
\end{figure}

\noindent\textbf{Case 3.}
(a)  Here $|\mathcal P|=1+q$, and the eigenline $E = \langle v_1\rangle$ is contained in every plane, so it is not adjacent to any plane. Moreover, every other line is contained with $E$ in a common plane of $\mathcal P$, so by Proposition \ref{dim3-derived2} (3) (c) it is adjacent to no line. Hence, $\deg(E) = 0$.
\smallskip

(b) Every plane in $\mathcal P$ contains the unique eigenline $\langle v_1\rangle$. If $P\in\mathcal P$, then $P$ is adjacent to the other $q$ planes and to all lines not contained in it. Therefore,
$\deg(P)=q+q^2=q^2+q$.
\smallskip

Now let $A \in \mathcal{L}$. If $\langle v_1\rangle \neq A\leq V$, then $A$ is contained only in the plane $V$. Hence it is non-adjacent to the other $q$ lines of $V$, and adjacent to all $q^2$ lines outside $V$. Also, it is adjacent to all planes except $V$, that is, to $q$ planes. Therefore, $\deg(A)=q^2+q$.
\smallskip

If $A\nleq V$, then $A$ is contained in a unique plane of $\mathcal P$, namely the plane generated by $A$ and $\langle v_1\rangle$. That plane contains exactly $q$ lines distinct from $A$, all of which are therefore non-adjacent to $A$. Thus, $A$ is adjacent to $(q^2+q)-q$ lines. It is also adjacent to all planes except the unique one containing it, that is, to $(q+1)-1$ planes. Hence again $\deg(A)=(q^2+q)-q + (q+1)-1 =q^2+q$.

\smallskip
\noindent\textbf{Case 4.}
(a) Here $|\mathcal P|=q^2+q+1$, and every two-dimensional subspace of $L$ belongs to $\mathcal P$. Each plane $P$ is adjacent to the other $q^2+q$ planes, and to all lines not contained in $P$, that is, to $q^2$ lines, so $\deg(P) = q^2+q+q^2 = 2q^2+q$.

(b) Every pair of distinct lines spans a two-dimensional subspace, which is a plane in $\mathcal{P}$. Hence, no two lines are adjacent. Further, every line $A$ is contained in exactly $q+1$ planes, so it is adjacent to $q^2$ planes, 
giving $\deg(A) = q^2$.
\end{proof}

\begin{figure}[H]
\centering
\begin{minipage}{0.45\textwidth}
\centering
\begin{tikzpicture}[
  scale=0.85,
  linev/.style={circle,draw=black,fill=yellow, minimum size=5pt,inner sep=0pt},
  planev/.style={circle,draw=black,fill=blue, minimum size=5pt,inner sep=0pt}
]
\def\r{2.1}
\node[planev] (V)  at (0:\r)   {};
\node[font=\small]      at (0:{\r+.4})   {$V$};
\node[linev]  (A1) at (40:\r)  {};
\node[font=\small]      at (40:{\r+.4})  {$L_1$};
\node[linev]  (A2) at (80:\r)  {};
\node[font=\small]      at (80:{\r+.4})  {$L_2$};
\node[planev] (P2) at (120:\r) {};
\node[font=\small]at (120:{\r+.45}) {$P_2$};
\node[linev]  (B1) at (160:\r) {};
\node[font=\small]   at (160:{\r+.45}) {$L_3$};
\node[linev]  (B2) at (200:\r) {};
\node[font=\small]      at (200:{\r+.4}) {$L_4$};
\node[planev] (P3) at (240:\r) {};
\node[font=\small]at (240:{\r+.4}) {$P_3$};
\node[linev]  (B3) at (280:\r) {};
\node[font=\small]     at (280:{\r+.45}) {$L_5$};
\node[linev]  (B4) at (320:\r) {};
\node[font=\small]at (320:{\r+.4}) {$L_6$};
\node[linev, black]  (E)  at (0,0) {};
\node[font=\tiny,anchor=west] at (-0.051,0) {$E$};
\draw[thick] (V)--(P2)  (V)--(P3) (P2)--(P3);
\draw[thick] (A1)--(B1) (A1)--(B2) (A1)--(B3) (A1)--(B4) (A1)--(P2) (A1)--(P3);
\draw[thick] (A2)--(B1)  (A2)--(B2) (A2)--(B3) (A2)--(B4) (A2)--(P2) (A2)--(P3);
\draw[thick] (B1)--(B3) (B1)--(B4) (B1)--(V) (B1)--(P3);
\draw[thick] (B2)--(B3) (B2)--(B4) (B2)--(V) (B2)--(P3);
\draw[thick] (B3)--(V) (B3)--(P2) (B4)--(V)  (B4)--(P2);
\end{tikzpicture}
\end{minipage} 
\hfil
\begin{minipage}{0.45\textwidth}
\centering

\begin{tikzpicture}[
scale=0.85,
transform shape,
linev/.style={circle, draw=black, fill=yellow, minimum size=7pt, inner sep=0pt},
planev/.style={circle, draw=black, fill=blue, minimum size=7pt, inner sep=0pt}
]

\def\r{2.2}
\def\labeldist{0.5}

\foreach \i/\a in {
1/0,
2/25.714,
3/51.428,
4/77.142,
5/102.857,
6/128.571,
7/154.285}
{
\node[linev] (L\i) at (\a:\r) {};
\node at (\a:{\r+\labeldist}) {$L_{\i}$};
}

\foreach \i/\a in {
1/180,
2/205.714,
3/231.428,
4/257.142,
5/282.857,
6/308.571,
7/334.285}
{
\node[planev] (P\i) at (\a:\r) {};
\node at (\a:{\r+\labeldist}) {$P_{\i}$};
}

\foreach \i in {1,...,7}
{
\foreach \j in {1,...,7}
{
\ifnum\i<\j
\draw[thick] (P\i)--(P\j);
\fi
}
}

\foreach \i in {1,...,7}
{
\foreach \j in {1,...,7}
{
\draw[thick] (L\i)--(P\j);
}
}
\end{tikzpicture}
\end{minipage}
\caption{The comaximal graphs for $q=2$ in cases 3 and 4}
\end{figure}
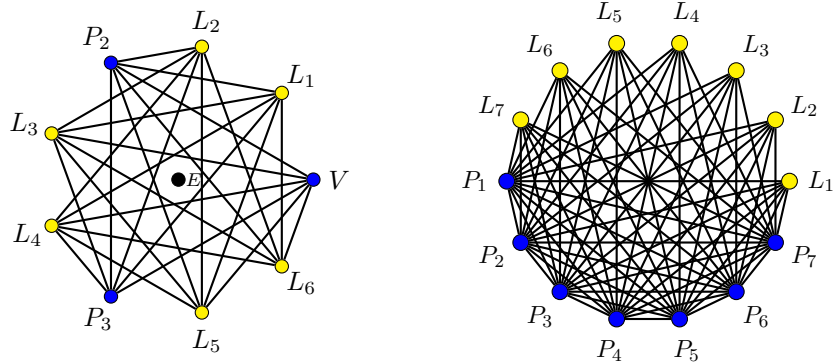

\textbf{Case 4. $\dim L'=3$ - the perfect case.} We would expect that the following result is well known, but we were unable to find a reference for it.

\begin{theorem} 
Let $L$ be a perfect three-dimensional Lie algebra over a finite field $\F_q$ where $q\neq 2$. Then $L\cong \mathfrak{sl}_2(\F_q)$ (the split case) or $L\cong \mathfrak{su}_2(\F_q)$. In the latter case, $L$ has a basis $e_1,e_2,e_3$ with $[e_1,e_2]=e_3$, $[e_2,e_3]=e_1$ and $[e_3,e_1]=e_2$ (the non-split case).
\end{theorem}

\begin{proof} 
Let $L$ be perfect. If $L$ is not simple, then it has a proper ideal $A$, and $A$, $L/A$ are both solvable. Hence $L$ is solvable and $L^2\neq L$, a contradiction. Thus, $L$ is simple, and so $L$ has a basis  $e_1, e_2, e_3$ with $[e_1,e_2]=e_3$, $[e_2,e_3]=\alpha e_1$ and $[e_3,e_1]=\beta e_2$, by \cite[Proposition 6.1, page 35]{sf}.
\par

Now, precisely half of the non-zero elements of $\F_q$ are quadratic residues modulo $q$. If $\beta$ is such a residue, then $-\beta$ is a non-residue and $x^2+\beta$ is irreducible; if $\beta$ is a non-residue, then $-\beta$ is a residue and $x^2+\beta$ is reducible. If $\beta$ is a non-residue, then we can make the basis change given in the last paragraph of \cite[page 35]{sf} to see that $L\cong \mathfrak{sl}_2(\F_q)$. If $\beta$ is a residue, then we can argue as in the real case on \cite[page 36]{sf} to see that $\beta^{1/2}$ and $\alpha^{1/2}$ exist, and putting $e_1'=\beta^{-1/2}e_1$, $e_2'=\alpha^{-1/2}e_2$ and $e_3'=(\alpha\beta)^{-1/2}e_3$ gives the basis claimed for $\mathfrak{su}_2(\F_q)$.
\end{proof}

From now on, we assume $q$ is odd. Every proper subalgebra of $\mathfrak{su}_2(\mathbb{F}_q)$ is one-dimensional, and so the graph is complete. The split case is more interesting. Fix the standard basis $\{x,y,h\}$ of $\mathfrak{sl}_2(\mathbb{F}_q)$ with Lie brackets given by $[h,x]=2x$, $[h,y]=-2y$, and $[x,y]=h$ (See Section 2). Assume further that $q$ is odd.

The structure of the subalgebra lattice of $\mathfrak{sl}_2(\F_q)$ is well known; we recall the parts relevant to the comaximal graph.

\begin{proposition}\label{sl2-subalgebras}
Let $L = \mathfrak{sl}_2(\F_q)$ with $q$ odd. Let $\mathcal{L}$ be its set of lines and $\mathcal{B}$ the set of $2$-dimensional Lie subalgebras. Then:
\begin{enumerate}
\item The set $\mathcal{L}$ has $q^2 + q + 1$ elements and every line is of exactly one of the following types:
\begin{enumerate}[(a)]
\item \emph{nilpotent}, generated by a nonzero nilpotent element, that is, a conjugate of $x$;
\item  \emph{split semisimple}, generated by a semisimple element diagonalizable over $\mathbb{F}_q$, that is, a 
        conjugate of $h$;
\item  \emph{nonsplit semisimple}, generated by a semisimple element not diagonalizable over $\mathbb{F}_q$, that is, elements whose characteristic polynomial has no roots in $\mathbb{F}_q$.
\end{enumerate}
Denoting the corresponding subsets by $\mathcal{N}$, $\mathcal{L}_s$, $\mathcal{L}_{ns}$ respectively, then the cardinalities are $|\mathcal{N}| = q+1$, $|\mathcal{L}_s| = \tfrac{q(q+1)}{2}$, and $|\mathcal{L}_{ns}| = \tfrac{q(q-1)}{2}$.
\item The $2$-dimensional Lie subalgebras of $L$ are precisely the Borel subalgebras. Moreover, $|\mathcal{B}| = q + 1$, and $\mathcal{B}$ is naturally in bijection with $\mathbb{P}^1(\mathbb{F}_q)$.
\end{enumerate}
\end{proposition}

\begin{proof}
(1) Every one-dimensional subspace of $L=\mathfrak{sl}_2(\F_q)$ is a Lie subalgebra, so $|\mathcal{L}| = q^2+q+1$.

Let $0 \neq u \in L$. Writing $u = ax + by + ch$, one computes that the characteristic polynomial of $u$ is $p_u(t) = t^2 - \Delta(u)$, where $\Delta(u) = c^2 + ab$. 
Thus, the type of the line $\langle u \rangle$ depends only on whether $\Delta(u)$ is zero, a nonzero square, or a nonsquare in $\F_q$.
\begin{enumerate}[(i)]
\item If $\Delta(u)=0$, then $u$ is nilpotent, and $\langle u \rangle$ is a nilpotent line. It is well known that all nonzero nilpotent elements of $\mathfrak{sl}_2(\F_q)$ are conjugate, and that the set of nilpotent lines is in bijection with the projective line $\mathbb{P}^1(\F_q)$. Hence $|\mathcal{N}| = q+1$.

\item If $\Delta(u)$ is a nonzero square, then $u$ is diagonalizable over $\F_q$, and $\langle u \rangle$ is a split semisimple line. If $\Delta(u)$ is a nonsquare, then $u$ is semisimple, but not diagonalizable over $\F_q$, and $\langle u \rangle$ is a nonsplit semisimple line. Since $q$ is odd, exactly half of the nonzero elements of $\F_q$ are squares and half are nonsquares, so the stated cardinalities follow.
\end{enumerate}
Thus $\mathcal{L} = \mathcal{N} \sqcup \mathcal{L}_s \sqcup \mathcal{L}_{ns}$, as claimed.

\smallskip

(2) The two-dimensional Lie subalgebras of $L$ are precisely the Borel subalgebras. Indeed, every $2$-dimensional subalgebra of $\mathfrak{sl}_2(\F_q)$ is solvable, and all maximal solvable subalgebras are conjugate to the standard Borel subalgebra $\langle H, x \rangle$. It follows that the set $\mathcal{B}$ of Borel subalgebras is in bijection with the projective line $\mathbb{P}^1(\F_q)$, and hence $|\mathcal{B}| = q+1$.
\end{proof}

\begin{proposition}\label{sl2-adjacency}
Let $L = \mathfrak{sl}_2(\F_q)$ with $q$ odd. Let $\mathcal{L}$ be the set of lines and $\mathcal{B}$ the set of Borel subalgebras. Then:
\begin{enumerate}
\item If $A, B \in \mathcal{B}$ with $A \neq B$, then $A \sim B$. Hence, the induced subgraph on $\mathcal{B}$ is the complete graph $K_{q+1}$.

\item If $A \in \mathcal{L}$ and $B \in \mathcal{B}$, then $A \sim B$ if and only if $A \nsubseteq B$.

\item If $A_1, A_2 \in \mathcal{L}$ with $A_1 \neq A_2$, then $A_1 \sim A_2$ if and only if there is no $B \in \mathcal{B}$ such that $A_1, A_2 \subseteq B$.
\end{enumerate}
\end{proposition}

\begin{proof}
(1) By Lemma \ref{g-adjacency} (2), the induced subgraph on $\mathcal{B}$ is complete.

\smallskip

\noindent (2) This follows immediately from Lemma~\ref{g-adjacency} (2).

\smallskip

\noindent (3) Let $A_1, A_2 \in \mathcal{L}$ with $A_1 \neq A_2$. Suppose that there exists a Borel subalgebra $B \in \mathcal{B}$ such that $A_1, A_2 \subseteq B$. Then $\langle A_1, A_2 \rangle \subseteq B \neq L$, so $A_1 \not\sim A_2$.

Conversely, suppose that there is no Borel subalgebra containing both $A_1$ and $A_2$. Then, clearly, $\langle A_1, A_2 \rangle = L$, and $A_1 \sim A_2$.
\end{proof}

\begin{proposition}
Using the standard basis $h,x,y$ for $L=sl_2(\F)$, where $[h,x]=2x$, $[h,y]=-2y$, $[x,y]=h$ (characteristic of $\F\neq 2$), the two-dimensional subalgebras are $B=\F x + \F h$ and $B(\alpha)=\F(h+\alpha x) + \F(y +\frac{1}{4} \alpha^2 x)$ with $\alpha \in \F$.
\end{proposition} 

\begin{proof} 
Let $S$ be a two-dimensional subalgebra of $L$. Suppose first that $x\in S$. Then there is an element $\alpha h+\beta y\in S$ and $[\alpha h+\beta y,x]=2\alpha x-\beta h\in S$. It follows that $S=\F x+\F h$.
\smallskip

So, suppose that $x\notin S$. Then $L=\F x+S$, so $h+\alpha x, y+\beta x\in S$ for some $\alpha,\beta\in \F$. Now $z=[h+\alpha x,y+\beta x]=-2y+2\beta x+\alpha h$, and this is in $S$ if and only if $z=A(h+\alpha x)+B(y+\beta x)$. Equating coefficients, we get $A=\alpha$,$B=-2$, and $2\beta=A\alpha+B\beta=\alpha^2-2\beta$. It follows that $S=F(h+\alpha x)+ F(y+\beta x)$ where $\beta=\frac{1}{4}\alpha^2$.
\end{proof}

So, when are lines connected? Two lines will be connected if and only if they don't both belong to the same Borel subalgebra. Consider the line $L=\lambda h+\mu x+\nu y$. There are two cases.
\bigskip

\noindent {\bf Case 1: $\lambda \neq 0$:}  We can write $L$ as $\F(h+\mu x+\nu y)$. Then $L\subset B$ if and only if $\nu=0$; in this case $L$ won't be connected to any other line in $B$.  
\par

Suppose $L\subset B(\alpha)$. Then $h+\mu x+\nu y=A(h+\alpha x) + B(y +\frac{1}{4} \alpha^2 x)$, whence $A=1$, $B=\nu$ and $\mu=\alpha+\frac{1}{4}\alpha^2\nu$. So $L\subset B(\alpha)$ if and only if $\nu\alpha^2+4\alpha-4\mu=0$. This last equation has a solution precisely when $\mu\nu+1$ is a square in $\F_q$; in this case, $L$ won't be connected to any other line in $B(\alpha)$.
\smallskip

\noindent {\bf Case 2: $\lambda=0$:}  Here $L=\mu x+\nu y$. This clearly is a subset of $B$ if and only if $\nu =0$; in this case, $L$ won't be connected to any other line in $B$. 
\par

It is a subset of $B(\alpha)$ if and only if $\mu x+\nu y=A(h+\alpha x) + B(y +\frac{1}{4} \alpha^2 x)$, whence $A=0$, $B=\nu$ and $\mu=\frac{1}{4}\alpha^2\nu$. So $L\subset B(\alpha)$ if and only if $\mu=0$ or $\alpha^2=\frac{4\mu}{\nu}$; in these cases $L$ won't be connected to any other line in $B(\alpha)$.
\smallskip

Nonsplit semisimple lines don't belong to any Borel subalgebra; nilpotent lines belong to a single Borel subalgebra; and split semisimple lines belong to precisely two Borel subalgebras.
\bigskip

\begin{example} 
The lines in $F=\F_3$ and the Borel subalgebras are classified in the following tables. 
\par


\begin{table}[H]
  \label{tab:rectas_sl2_F3}
  \centering
  \caption{Classification of lines in $\mathfrak{sl}_2(\F_3)$ according to the discriminant $\Delta = c^2 + ab$, where the generator is $ax + by + ch$.}
  \begin{tabular}{l|l|c|l}
    \hline
    \textbf{Line} & \textbf{Generator} & $\Delta$ & \textbf{Type} \\
    \hline
    $L_1  = \F_3 h$           & $h$  & $1$ & Split      \\
    $L_2  = \F_3 x$           & $x$  & $0$ & Nilpotent \\
    $L_3  = \F_3 y$           & $y$  & $0$  & Nilpotent \\
    $L_4  = \F_3(h+x)$       & $h+x$  & $1$  & Split      \\
    $L_5  = \F_3(h+y)$       & $y+h$  & $1$  & Split      \\
    $L_6  = \F_3(x+y)$       & $x+y$  & $1$  & Split      \\
    $L_7  = \F_3(h+2x)$      & $2x+h$  & $1$  & Split      \\
    $L_8  = \F_3(h+2y)$      & $2y+h$  & $1$  & Split      \\
    $L_9  = \F_3(x+2y)$      & $x+2y$  & $2$  & Nonsplit   \\
    $L_{10}= \F_3(x+y+h)$    & $x+y+h$  & $2$  & Nonsplit   \\
    $L_{11}= \F_3(2x+y+h)$   & $2x+y+h$  & $0$ & Nilpotent \\
    $L_{12}= \F_3(x+2y+h)$   & $x+2y+h$  & $0$ & Nilpotent \\
    $L_{13}= \F_3(x+y+2h)$   & $x+y+2h$  & $2$ & Nonsplit   \\
    \hline
  \end{tabular}
\end{table}
\smallskip





\begin{table}[H]
  \label{tab:borels_sl2_F3}
  \centering
  \caption{Borel subalgebras of $\mathfrak{sl}_2(\F_3)$ and the lines they contain.}
  \begin{tabular}{l|c}
    \hline
   \textbf{Borel}
      & \textbf{Contained lines}  \\
    \hline
    $B=\F_3 x+\F_3 h$  & $L_1,\, L_2,\, L_4,\, L_7$     \\
    $B(0)=\F_3 h+\F_3y $ & $L_1,\, L_3,\, L_5,\, L_8$    \\
    $B(1)=\F_3(h+x)+\F_3(y+x)$ & $L_4,\, L_6,\, L_8,\, L_{11}$     \\
    $B(2)=\F_3(h+2x)+\F_3(y+x)$ & $L_5,\, L_6,\, L_7,\, L_{12}$   \\
    \hline
  \end{tabular}
\end{table}

So, each of $L_9, L_{10}$ and $L_{13}$ is adjacent to the other 12 lines and to all 4 Borel subalgebras, as they are not in any Borel subalgebra. These are the nonsplit semisimple lines. The lines $L_2, L_3, L_{11}$ and $L_{12}$ are each adjacent to 9 other lines and to 3 Borel subalgebras, as they are in a single Borel subalgebra. These are the nilpotent lines. The lines $L_1, L_4, L_5, L_6, L_7$ and $L_8$ are each adjacent to 6 other lines and to 2 Borel subalgebras, as they are in 2 Borel subalgebras. These are the split semisimple lines.
\end{example} 

\begin{figure}[H]
    \label{sl2-perfect}
\begin{center}
\begin{tikzpicture}[
  lns/.style   = {circle, fill=black,           minimum size=7pt, inner sep=0pt},
  borel/.style = {circle, fill=red,             minimum size=7pt, inner sep=0pt},
  nilp/.style  = {circle, fill=green!70!black,  minimum size=7pt, inner sep=0pt},
  split/.style = {circle, fill=blue,            minimum size=7pt, inner sep=0pt},
  edge/.style  = {draw=black,thick}]

\def\rLns{3.5}    
\def\rBorel{3.5}  
\def\rNilp{3.5}   
\def\rSplit{3.5}  


\node[lns,   label={[font=\small]above:$L_9$}]   (L9)  at (40 :\rLns)  {};
\node[lns,   label={[font=\small]below left:$L_{10}$}]  (L10) at (160:\rLns)  {};
\node[lns,   label={[font=\small]below right:$L_{13}$}] (L13) at (280:\rLns)  {};

\node[borel, label={[font=\small]above:$\mathcal{B}$}]       (B)  at (100 :\rBorel) {};
\node[borel, label={[font=\small]left:$\mathcal{B}(0)$}]     (B0) at (180:\rBorel) {};
\node[borel, label={[font=\small]below:$\mathcal{B}(1)$}]    (B1) at (260:\rBorel) {};
\node[borel, label={[font=\small]right:$\mathcal{B}(2)$}]    (B2) at (0  :\rBorel) {};

\node[nilp,  label={[font=\small]above right:$L_2$}]    (L2)  at (60 :\rNilp) {};
\node[nilp,  label={[font=\small]above left:$L_3$}]     (L3)  at (120:\rNilp) {};
\node[nilp,  label={[font=\small]below left:$L_{11}$}]  (L11) at (220:\rNilp) {};
\node[nilp,  label={[font=\small]below right:$L_{12}$}] (L12) at (300:\rNilp) {};

\node[split, label={[font=\small]above:$L_1$}]        (L1) at (80 :\rSplit) {};
\node[split, label={[font=\small]above right:$L_4$}]  (L4) at (20 :\rSplit) {};
\node[split, label={[font=\small]below right:$L_5$}]  (L5) at (320:\rSplit) {};
\node[split, label={[font=\small]below:$L_6$}]        (L6) at (240:\rSplit) {};
\node[split, label={[font=\small]below left:$L_7$}]   (L7) at (200:\rSplit) {};
\node[split, label={[font=\small]above left:$L_8$}]   (L8) at (140:\rSplit) {};


\draw[edge] (B)--(B0)--(B1) --(B2) -- (B0);   
\draw[edge] (B1)--(B)--(B2); 

 
\draw[edge] (B2) --(L1)--(B1);   
\draw[edge] (B1)--(L2)--(B0) (L2)--(B2);
\draw[edge] (L3)--(B) (B2) -- (L3)--(B1) ;
\draw[edge] (B2) -- (L4)--(B0);  
\draw[edge] (L5)--(B) (L5)--(B1);
\draw[edge] (L6)--(B) (L6)--(B0);
\draw[edge] (L7)--(B0) (L7)--(B1);
\draw[edge] (L8)--(B) (L8)--(B2);
\draw[edge] (L9)--(B) (L9)--(B0);
\draw[edge] (B2) -- (L9)--(B1) ;
\draw[edge] (L10)--(B) (L10)--(B0);
\draw[edge] (L10)--(B1) (L10)--(B2);
\draw[edge] (L11)--(B) (L11)--(B0) (L11)--(B2);
\draw[edge] (L12)--(B) (L12)--(B0) (L12)--(B1);
\draw[edge] (L13)--(B) (L13)--(B0);
\draw[edge] (L13)--(B1) (L13)--(B2);

%

\draw[edge] (L9)--(L1) (L9)--(L2) (L9)--(L3);
\draw[edge] (L9)--(L4) (L9)--(L5) (L9)--(L6);
\draw[edge] (L9)--(L7) (L9)--(L8);
\draw[edge] (L9)--(L10) (L9)--(L11);
\draw[edge] (L9)--(L12) (L9)--(L13);

\draw[edge] (L10)--(L1) (L10)--(L2) (L10)--(L3);
\draw[edge] (L10)--(L4) (L10)--(L5) (L10)--(L6);
\draw[edge] (L10)--(L7) (L10)--(L8);
\draw[edge] (L10)--(L11) (L10)--(L12) (L10)--(L13);

\draw[edge] (L13)--(L1) (L13)--(L2) (L13)--(L3);
\draw[edge] (L13)--(L4) (L13)--(L5) (L13)--(L6);
\draw[edge] (L13)--(L7) (L13)--(L8);
\draw[edge] (L13)--(L11) (L13)--(L12);

\draw[edge] (L1)--(L6) (L1)--(L11) (L1)--(L12);

\draw[edge] (L2)--(L3) (L2)--(L5) (L2)--(L6);
\draw[edge] (L2)--(L8) (L2)--(L11) (L2)--(L12);

\draw[edge] (L3)--(L4) (L3)--(L6) (L3)--(L7);
\draw[edge] (L3)--(L11) (L3)--(L12);

\draw[edge] (L4)--(L5) (L4)--(L12);

\draw[edge] (L5)--(L11);

\draw[edge] (L7)--(L8) (L7)--(L11);

\draw[edge] (L8)--(L12);

\draw[edge] (L11)--(L12);
\end{tikzpicture}
\end{center}
  
\caption{The comaximal graph of $\mathfrak{sl}_2(\mathbb{F}_3)$. Blue vertices are split semisimple lines, red vertices are Borels, green vertices are nilpotent lines, and black vertices are nonsplit semisimple lines.}
\end{figure}
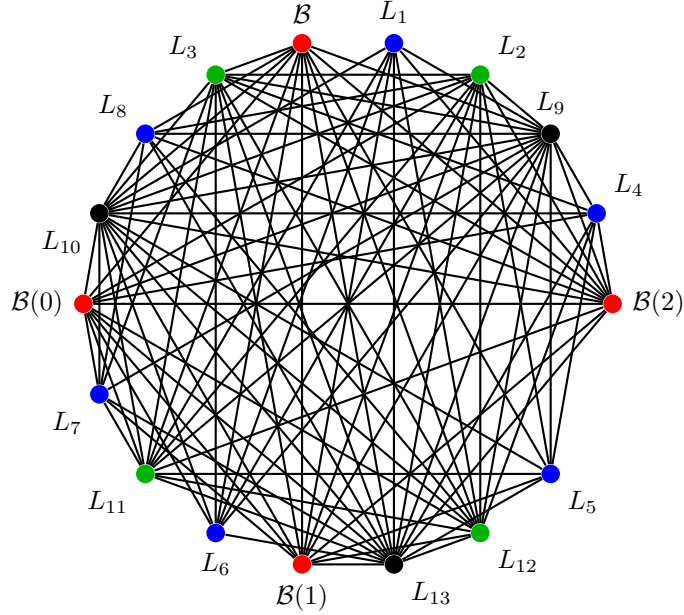

The following result is well-known.

\begin{lemma}\label{borel-lines}
Each Borel subalgebra of $\mathfrak{sl}_2(\mathbb{F}_q)$ contains exactly one nilpotent line and exactly $q$ split semisimple lines.
\end{lemma}

\begin{proof}
Simply note that they are all conjugate to the standard upper triangular one.
\end{proof}

\begin{corollary}\label{sl2-degrees}
Let $L = \mathfrak{sl}_2(\mathbb{F}_q)$ with $q$ odd. Then:
\begin{enumerate}
    \item For every $B \in \mathcal{B}$, $\deg(B) = q^2+q$.
    \item For $A \in \mathcal{L}$,
    \[\deg(A) =
        \begin{cases}
            q^2+q  & \text{if } A \in \mathcal{N},   \\
            q^2 - 1 & \text{if } A \in \mathcal{L}_s, \\
            (q+1)^2 & \text{if } A \in \mathcal{L}_{ns}.
        \end{cases}    \]
In particular, every Borel subalgebra of $L$ and every nilpotent line have the same degree $q^2+q$.
    \item The degree sequence of $\Gamma(L)$ is
    \[\Bigl(\underbrace{q^2+q,\ldots, q^2+q}_{2(q+1)}, \underbrace{q^2-1,\ldots,q^2-1}_{\frac{q(q+1)}{2}}, \underbrace{(q+1)^2,\ldots,(q+1)^2}_{\frac{q(q-1)}{2}}\Bigr).    \]
    In particular, $\Gamma(L)$ is not regular for $q \geq 3$.
    \item $\Gamma(L)$ has order $(q+1)^2+1$ and size $\frac{q(q+1)^3}{2}$.
\end{enumerate}
\end{corollary}

\begin{proof}
(1) Each $B \in \mathcal{B}$ is adjacent to the other $q$ Borels. Since $B$ contains exactly $q+1$ lines, it is adjacent to the remaining $q^2$ lines. Hence $\deg(B) = q + q^2$.

\smallskip
(2) Let $A\in \mathcal{L}$. Then the neighbors of $A$ consist of the $(q+1)-m(A)$ Borels not containing $A$, plus all lines $B$ such that no Borel contains both $A$ and $B$. We differentiate the three possible cases.

\smallskip
\noindent\textit{Case 1: $A \in \mathcal{N}$.} There is a unique Borel $B_A$ containing $A$, so $A$ is adjacent to $q$ 
Borels. By Lemma \ref{borel-lines} the $q$ lines of $B_A \setminus \{A\}$ are all split semisimple. These are precisely the lines sharing a Borel with $A$, hence the lines not adjacent to $A$. Therefore $A$ is adjacent 
to $(q^2+q)-q$ lines, and so $\deg(A) = q + q^2$.

\smallskip
\noindent\textit{Case 2: $A \in \mathcal{L}_s$.} There are exactly two Borels $B_1, B_2$ containing $A$. Since their intersection is the single line $A$, then $B_1$ and $B_2$ share no other line. The lines not adjacent to $A$ are those sharing a Borel with $A$, namely the lines of $B_1 \setminus \{A\}$ and $B_2 \setminus \{A\}$. These two sets are disjoint, and each has size $q$, giving $2q$ lines not adjacent to $A$. Hence $A$ is adjacent to $(q^2+q)-2q$ lines and to $q-1$ Borels, giving $\deg(A) = (q-1) + (q^2+q)-2q = q^2-1$.

\smallskip
\noindent\textit{Case 3: $A \in \mathcal{L}_{ns}$.} The line $A$ is contained in no Borel. Hence, for any other line $B$, no Borel contains both $A$ and $B$, so $A \sim B$. Thus $A$ is adjacent to all $q+1$ Borels and to all $q^2+q$ other lines, and we have $\deg(A) = (q+1) + (q^2+q) = (q+1)^2$.

\smallskip
(3) The degree sequence follows directly from (1) - (2) and the cardinalities $|\mathcal{B}| = |\mathcal{N}| = q+1$, 
$|\mathcal{L}_s| = \frac{q(q+1)}{2}$ and $|\mathcal{L}_{ns}| = \frac{q(q-1)}{2}$ from Proposition \ref{sl2-subalgebras}. The three 
degree values $q(q+1)$, $q^2-1$ and $(q+1)^2$ are pairwise distinct 
for $q \geq 3$, so $\Gamma(L)$ is not regular.

\smallskip
(4) The order of $\Gamma(L)$ is $|\mathcal{B}|+|\mathcal{L}| = (q+1)+(q^2+q+1) = (q+1)^2+1$. For the size of $\Gamma(L)$, we apply the handshaking lemma:
\begin{align*}
2|E| &= |\mathcal{B}|(q^2+q) + |\mathcal{N}| (q^2+q) + |\mathcal{L}_s| (q^2-1) + |\mathcal{L}_{ns}| (q+1)^2 \\
     &= (q+1)(q^2+q) + (q+1)(q^2+q) + \tfrac{q^2+q}{2}(q^2-1) + \tfrac{q^2+q}{2}(q+1)^2 \\
     &= q(q+1)^3.
\end{align*}
Therefore, $|E| = \dfrac{q(q+1)^3}{2}$.
\end{proof}

\begin{corollary}\label{sl2-basic-invariants}
Let $L = \mathfrak{sl}_2(\mathbb{F}_q)$ with $q$ odd. Then:
\begin{enumerate}
    \item $\Gamma(L)$ is connected
    \item $\mathrm{girth}(\Gamma(L)) = 3$
    \item $\omega(\Gamma(L)) = \tfrac{q^2+q+2}{2}$
    \item $\mathrm{diam}(\Gamma(L)) = 2$
    \item $\gamma(\Gamma(L)) = 1$
    \item Every vertex of $\Gamma(L)$ lies on a triangle
    \item $\Gamma(L)$ is non-planar
    \item $r(\Gamma(L)) = 1$. Moreover, the center of $\Gamma(L)$ is the set of all nonsplit semisimple lines.
\end{enumerate}
\end{corollary}

\begin{proof}
(1) - (2) Since $\Gamma(L)[\mathcal{B}] \cong K_{q+1}$, and every vertex of $\Gamma(L)$ is adjacent to at least one Borel, we have that $\Gamma(L)$ is connected. Further, as $q+1 \geq 4$, the graph $\Gamma(L)$ contains triangles, so $\mathrm{girth}(\Gamma(L)) = 3$.

\smallskip
\noindent (3) We first show that $\mathcal{L}_{ns} \cup \mathcal{B}$ is a clique of size $\frac{q^2+q+2}{2}$. Any two Borels are adjacent. For $A \in \mathcal{L}_{ns}$ and $B \in \mathcal{B}$, we have $A \not\subseteq B$, so $A \sim B$. For two distinct lines $A_1, A_2 \in \mathcal{L}_{ns}$, no Borel contains either, so no Borel contains both, and $A_1 \sim A_2$. Hence $\mathcal{L}_{ns} \cup \mathcal{B}$ is a 
clique of size
\[|\mathcal{L}_{ns}| + |\mathcal{B}| = \tfrac{q(q-1)}{2} + (q+1) = \tfrac{q^2+q+2}{2}.\]
No larger clique exists. In fact, a nilpotent line $A \in \mathcal{N}$ lies in exactly one Borel of $\mathcal{B}$, so it is not adjacent to that Borel and cannot extend the clique. Further, a split semisimple line $A \in \mathcal{L}_s$ lies in exactly two Borels of $\mathcal{B}$, so it is
not adjacent to those two Borels and cannot extend the clique either. Hence $\omega(\Gamma(L)) = \tfrac{q^2+q+2}{2}$.
\smallskip

\noindent (4) By Corollary \ref{sl2-degrees}, every $A \in \mathcal{L}_{ns}$ has degree $(q+1)^2 = |V(\Gamma(L))|-1$, so $A$ is adjacent to every other vertex of $\Gamma(L)$. Hence, for any two non-adjacent vertices $C, D$, the path $C\sim A \sim D$ has length $2$. Since $\Gamma(L)$ is not complete, the diameter of $\Gamma(L)$ is exactly 2.
\smallskip

\noindent (5) Using (4), we have that any $A \in \mathcal{L}_{ns}$ is adjacent to every other vertex of $\Gamma(L)$. Hence $\{A\}$ is a dominating set of size $1$, so $\gamma(\Gamma(L)) = 1$.

\noindent (6) Let $C$ be any vertex of $\Gamma(L)$. Since $|\mathcal{L}_{ns}| \geq 3$ for $q\geq 3$, for two distinct lines $A_1, A_2 \in \mathcal{L}_{ns}$ we have the triangle $C\sim A_1 \sim A_2\sim C$.
\smallskip

\noindent (7) Since $q \geq 3$, we have $|\mathcal{L}_{ns}| \geq 3$ and $|\mathcal{B}| \geq 4$. By
(3) $\mathcal{L}_{ns} \cup \mathcal{B}$ is a clique of size at least 7. Then the subgraph induced on $\mathcal{L}_{ns} \cup \mathcal{B}$ contains $K_5$. By Kuratowski's theorem, $\Gamma(L)$ is non-planar.
\smallskip

\noindent (8) For every $A \in \mathcal{L}_{ns}$ hold that $\mathrm{ecc}(A) = 1$. Then in particular $r(\Gamma(L)) \leq 1$. Since $\Gamma(L)$ is not a complete graph, we have that 
$r(\Gamma(L)) \geq 1$, hence $r(\Gamma(L)) = 1$. It remains to show that no other vertex has eccentricity $1$. Since $\mathrm{diam}(\Gamma(L)) = 2$ by (4), every vertex has eccentricity at most $2$. A vertex $C$ has eccentricity $1$ if and only if it is adjacent to all other vertices; this holds precisely for all elements in $\mathcal{L}_{ns}$. Thus, the center of $\Gamma(L)$ is $\mathcal{L}_{ns}$.
\end{proof}

\begin{theorem}\label{chromatic-sl2}
Let $L = \mathfrak{sl}_2(\mathbb{F}_q)$ with $q$ odd. Then $\chi(\Gamma(L)) = \omega(\Gamma(L)) = \tfrac{q^2+q+2}{2}$.
\end{theorem}

\begin{proof}
Since $\chi(\Gamma(L)) \geq \omega(\Gamma(L))$ always holds, and $\omega(\Gamma(L)) = \frac{q^2+q+2}{2}$ by Corollary \ref{sl2-basic-invariants} (3), it suffices to construct a proper coloring using exactly 
$\frac{q^2+q+2}{2}$ colors. By Corollary \ref{sl2-basic-invariants} (3), the set $L_{ns} \cup \mathcal{B}$ is a clique of size $\frac{q^2+q+2}{2}$, so any proper coloring must assign pairwise distinct colors to its vertices. We use this as the 
basis for our coloring. Label the $q+1$ Borel subalgebras as $B_1, \ldots, B_{q+1}$ and the $\frac{q(q-1)}{2}$ nonsplit semisimple lines as $A_1, \ldots, A_{\frac{q(q-1)}{2}}$. Define a coloring with color set $\{1, \ldots, \tfrac{q^2+q+2}{2}\}$ as follows:
\begin{enumerate}
\item Assign color $i$ to $B_i$, for $i = 1, \ldots, q+1$.

\item Assign color $q+1+j$ to $A_j$, for $j = 1, \ldots,     \frac{q(q-1)}{2}$.
    
\item Every remaining vertex is a nilpotent or split semisimple line. A nilpotent line is contained in exactly one Borel; assign it the color of that Borel. A split semisimple line is contained in exactly two Borels; choose either one and assign its color.
\end{enumerate}
It remains to verify that no two vertices sharing a Borel color $i$ are adjacent. Every vertex with color $i$ is contained in $B_i$, and by Proposition \ref{sl2-adjacency} (3), two vertices sharing a common Borel are non-adjacent. Hence, vertices with the same Borel color $i$ form an independent set, and the coloring is proper.
\end{proof}

\bibliographystyle{plane}

\end{document}